\documentclass[]{interact}

\usepackage{epstopdf}

\usepackage[caption=false]{subfig}

\usepackage[longnamesfirst,sort]{natbib}
\bibpunct[, ]{(}{)}{;}{a}{,}{,}

\theoremstyle{plain}
\newtheorem{theorem}{Theorem}[section]
\newtheorem{lemma}[theorem]{Lemma}
\newtheorem{corollary}[theorem]{Corollary}

\usepackage{color}
\theoremstyle{definition}
\newtheorem{definition}[theorem]{Definition}
\newtheorem{example}[theorem]{Example}
\newtheorem{assumption}[theorem]{Assumption}
\theoremstyle{remark}
\newtheorem{remark}{Remark}

\begin{document}


\title{Generalized Location-Scale Mixtures of Elliptical Distributions: Definitions and Stochastic Comparisons}

\author{
\name{Tong Pu\textsuperscript{a,b}, Yiying Zhang\textsuperscript{a} and Chuancun Yin\textsuperscript{b}\thanks{Corresponding author: Chuancun Yin. Email: {ccyin@qfnu.edu.cn}}}
\affil{\textsuperscript{a}Department of Mathematics, Southern University of Science and Technology, Shenzhen 518055, Guangdong, China; \textsuperscript{b}School of Statistics and Data Science, Qufu Normal University, Qufu 273165, Shandong, China}}

\maketitle

\begin{abstract}
  This paper proposes a unified class of generalized location-scale mixture of multivariate elliptical distributions and studies integral stochastic orderings of random vectors following such distributions. Given a random vector $\boldsymbol{Z}$, independent of $\boldsymbol{X}$ and $\boldsymbol{Y}$, the scale parameter of this class of distributions is mixed with a function $\alpha(\boldsymbol{Z})$ and its skew parameter is mixed with another function $\beta(\boldsymbol{Z})$. Sufficient (and necessary) conditions are established for stochastically comparing different random vectors stemming from this class of distributions by means of several stochastic orders including the usual stochastic order, convex order, increasing convex order, supermodular order, and some related linear orders. Two insightful assumptions for the density generators of elliptical distributions, aiming to control the generators' tail, are provided to make stochastic comparisons among mixed-elliptical vectors. Some applications in applied probability and actuarial science are also provided as illustrations on the main findings.
\end{abstract}


\section{Introduction and Motivation} \label{introduction}
Stochastic orders are partial orders and serve as a powerful tool for comparing different interested random variables. Stochastic orderings have found numerous applications in various of research fields like statistics and probability (\cite{Cal2006stochastic}), actuarial science (\cite{denuit2006actuarial}, \cite
{Bauerle2014}), and operations research (\cite{fabian2011Processing}). Different kinds of stochastic orders have different properties, meanings and applications. Interested readers may refer to \cite{denuit2006actuarial}, \cite{muller2002comparison} and \cite{shaked2007stochastic} for more details.


\par Many stochastic orders can be characterized by the integral stochastic orders, seeking for orderings between random vectors $\boldsymbol{X}$ and $\boldsymbol{Y}$ by comparing $Ef(\boldsymbol{X})$ and $Ef(\boldsymbol{Y})$, where $f \in \boldsymbol{F}$ and $\boldsymbol{F}$ is a certain class of functions. An insightful treatment for this class of orders can be found in \cite{muller2001stochastic}, who provided necessary and sufficient conditions for stochastic ordering results of multivariate normal distributions. They firstly established an identity for $Ef(\boldsymbol{Y}) - Ef(\boldsymbol{X})$, and then derived sufficient conditions for various stochastic orderings by using this identity. \cite{ding2004some} extended these results to Kotz-type distributions which form a special class of elliptical symmetric distributions. In recent years, some other integral stochastic orderings of multivariate elliptical distribution have been studied by many researchers such as \cite{yin2019stochastic} (by aforementioned identity) and \cite{Ansari2020Ordering} (by using pure probabilistic approaches). For relevant papers for other distributions, we refer the reader to the study of stochastic orderings of skew-normal distributions \citep{jamali2020comparison}, multivariate normal mean-variance mixtures \citep{jamali2020integral}, skew-normal scale-shape mixtures \citep{jamali2020integral}, scale mixtures of the multivariate skew-normal distributions \citep{Amiri2020Linear}, and matrix variate skew-normal distributions \citep{pu2022identity}.

\par Elliptical distributions, which can be seen as convenient extensions of multivariate normal distributions, were introduced by \cite{Kelker1970Distribution}. Some properties and characterizations of this family of distributions were discussed in \cite{Fang1990}. Elliptical distributions provide an attractive tool for modeling many practical scenarios in statistics, economics, finance and actuarial science since they can describe fat or light tails of distributions due to the flexibility of the density functions. Interested readers are referred to the three monographs of \cite{Fang1990}, \cite{Gupta2013} and \cite{McNeil2015}. In the literature, many interesting stochastic comparison results have been established for random vectors with elliptical distributions. For example, \cite{Davidov2013linear} showed an important result that the positive linear usual stochastic order coincides with the multivariate usual stochastic order for elliptically distributed random vectors. Some sufficient conditions were obtained in \cite{landsman2006stochastic} for comparing random vectors having bivariate elliptical distributions in the sense of the convex order, the increasing convex order and the concordance order. \cite{pan2016stochastic} studied the convex and increasing convex orderings of multivariate elliptical random vectors and derived some necessary and sufficient conditions.

\par However, the class of elliptical distributions fails to capture the skewness of data because of its symmetrical characteristic. Many researchers tries to generalize the class of elliptical distributions to fill this gap. It is well known that the scale mixture of elliptical distributions is still elliptical; therefore, adding a factor of skewness into the elliptical model seems to be a straightforward method. \cite{Barndorff1982Normal} introduced the mean-variance mixture of multinormal distributions generated by the following stochastic representation
\begin{equation} \label{NMVM-rep}
\boldsymbol{Y} \overset{d}{=} \boldsymbol{\mu} +\sqrt{\boldsymbol{Z}} \boldsymbol{X} + \boldsymbol{Z}\boldsymbol{\delta},
\end{equation}
where $\boldsymbol{\mu}, \boldsymbol{\delta} \in \mathbb{R}^n$, $\boldsymbol{X} \sim N_n(0,\boldsymbol{\Sigma})$ and $\boldsymbol{Z}$ is a non-negative random quantity. This class of distributions plays an important role in statistical modeling; see, for example, \cite{Jones2004}, \cite{kim2019Tail} and \cite{jamali2020comparison}. Besides, one can easily find that it is a natural extension to assume that $\boldsymbol{X}$ is an elliptical distributed vector. With this setting of $\boldsymbol{X}$, the distribution presented in (\ref{NMVM-rep}) becomes the location-scale mixture of elliptical distributions. Some interesting basic properties and applications about the location-scale mixture of elliptical distributions can be found in \cite{zuo2021}.

\par An alternative way to take skewness into consideration was introduced by \cite{Arslan2008} through considering the following stochastic representation
\begin{equation} \label{GHSS-rep}
\boldsymbol{Y} \overset{d}{=} \boldsymbol{\mu} +\frac{1}{\sqrt{\boldsymbol{Z}}} \boldsymbol{X} + \frac{1}{\boldsymbol{Z}}\boldsymbol{\delta},
\end{equation}
where $\boldsymbol{X}$ follows elliptical distribution with location parameter $0$ and scale parameter $\boldsymbol{\Sigma}$ and $\boldsymbol{Z}$ is a non-negative random quantity follows beta distribution. This class of distributions is termed as the ``Generalized Hyperbolic Skew-Slash Distributions''.

\par In finance area, \cite{simaan1993} argued that the vector of returns on a set of concerned financial assets should be represented as
\begin{equation} \label{fina-rep}
\boldsymbol{Y} \overset{d}{=} \boldsymbol{\mu} + \boldsymbol{X} + \boldsymbol{Z}\boldsymbol{\delta},
\end{equation}
where $\boldsymbol{X}$ follows elliptical distribution with location parameter $0$ and scale parameter $\boldsymbol{\Sigma}$ and $\boldsymbol{Z}$ is a non-negative random quantity. It is worth noting that, in models (\ref{NMVM-rep}) and (\ref{GHSS-rep}), both location and scale parameters are mixed with the same positive random variable $\boldsymbol{Z}$, while in model (\ref{fina-rep}) only the location parameter is mixed with $\boldsymbol{Z}$. It is clear that the class in (\ref{fina-rep}) cannot be obtained from the class in (\ref{NMVM-rep}) nor the class in (\ref{GHSS-rep}). \cite{Adcock2012} studied some special cases of (\ref{fina-rep}), including the
normal-exponential and normal-gamma distributions. The author also presented some applications of this class of distributions in capital pricing, returns on financial assets and portfolio selections.

\par The univariate and multivariate skew-normal distributions were introduced in \cite{azzalini1985class} and \cite{azzalini1996multivariate}. An $n$-dimensional random vector $\boldsymbol{Y}$ is said to follow multivariate skew-normal distribution if it has stochastic representation
\begin{equation} \label{SN-rep}
\boldsymbol{Y} \overset{d}{=} \boldsymbol{\mu} + \boldsymbol{\sigma} \left( \boldsymbol{\delta Z} + \boldsymbol{X} \right),
\end{equation}
where $\boldsymbol{\delta} = \left( \delta_1,\delta_2, \dots, \delta_n\right)^T$, $-1 < \delta_i<1$ for all $1 \leq i \leq n$, $\boldsymbol{X} \sim N_n\left(0,\overline{\boldsymbol{\Sigma}} - \boldsymbol{\delta}\boldsymbol{\delta}^T\right)$, and the random quantity $\boldsymbol{Z}$, independent of $\boldsymbol{X}$, has a standard normal distribution within the truncated interval $(0, +\infty)$. Here, the square matrix $\boldsymbol{\sigma}$ is a diagonal matrix formed by $\boldsymbol{\sigma} = \left( \boldsymbol{\Sigma} \odot \mathbf{I}_n\right)^{1/2}$, where $\odot$ stands for the Hadamard product. \cite{Azzalini2005} illustrated various areas of application of skew-normal distribution, including selective sampling, models for compositional data, robust methods and non-linear time series. Recently, a general new family of the mixture of multivariate normal distributions was introduced by \cite{Negarestani2019} and \cite{madadi2020family} based on arbitrary random variable $\boldsymbol{Z}$ in (\ref{SN-rep}). In addition to the aforementioned approaches to model skewed data, \cite{Branco2001}, \cite{Arnold2002}, \cite{Wang2004} and \cite{Dey2005} extended elliptical distributions to skew-elliptical distributions under different perspectives. These extensions are mainly based on the density functions, correlation of random vectors and conditional representations.

\par Inspired by the stochastic representations (\ref{NMVM-rep}), (\ref{GHSS-rep}), (\ref{fina-rep}) and (\ref{SN-rep}), we propose a unified method to introduce the so-called class of generalized location-scale mixture of elliptical (GLSE) distributions, which takes skewness into consideration and gives a mathematically tractable extension of the multivariate elliptical distribution. Aforementioned four classes of distributions are special cases of GLSE distributions. Furthermore, we derive some sufficient and necessary conditions for various integral stochastic orderings of random vectors following the GLSE distributions, where we shall apply some common technical tricks used in \cite{muller2001stochastic}, \cite{yin2019stochastic} and \cite{jamali2020comparison} to the GLSE distributions.


\par The rest of the paper is organized as follows. In Section \ref{preliminaries}, we review multivariate elliptical distribution and recall some key properties and characterizations. We also present a brief review of various integral stochastic orderings. In Section \ref{lse}, we introduce the class of GLSE distributions and present some related properties. Section \ref{results} establishes necessary and/or sufficient conditions for integral stochastic orderings and presents some actuarial applications. Some applications on stochastic ordering results  of important quantities for individual and collective risk models are provided in Section \ref{Applications}. Section \ref{remarks} concludes with a short discussion and some possible directions for future research.
\section{Preliminaries}\label{preliminaries}
We will use lowercase letters, bold lowercase letters, bold capital letters and bold Italic letters to denote numbers, vectors, matrices and random vectors respectively. Let $\Phi \left( \cdot \right)$ and $\phi \left( \cdot \right)$ denote the cumulative distribution function (CDF) and probability density function (PDF) of the univariate standard normal distribution, respectively, and $\Phi_n \left( \cdot ;\boldsymbol{\mu}, \boldsymbol{\Sigma}\right)$ and $\phi_n \left( \cdot ;\boldsymbol{\mu}, \boldsymbol{\Sigma} \right)$ denote the cumulative distribution function and probability density function of $n$-dimensional normal distribution with mean vector $\boldsymbol{\mu}$ and covariance matrix $\boldsymbol{\Sigma}$.

\par For twice continuously differentiable function $f: \mathbb{R}^n \to \mathbb{R}$, we use
\begin{equation} \nonumber
\nabla f(\mathbf{x}) = \left( \frac{\partial}{\partial \mathbf{x}_i} f(\mathbf{x}) \right)^n_{i=1}, \quad H_f(\mathbf{x}) = \left( \frac{\partial^2}{\partial \mathbf{x}_i \partial \mathbf{x}_j } f(\mathbf{x}) \right)^n_{i,j=1}
\end{equation}
to denote the gradient vector and the Hessian matrix of $f$, respectively. We use ${\rm tr}(\mathbf{C})$ to denote the trace of square matrix $\mathbf{C}$. For $n$-dimensional vectors $\boldsymbol{a}$ and $\boldsymbol{b}$, their inner product is denoted as $\langle \boldsymbol{a}, \boldsymbol{b}\rangle = \boldsymbol{a}^T\boldsymbol{b}$. For $n \times p$-dimensional matrices $\mathbf{A}$ and $\mathbf{B}$, their inner product is expressed as $\langle \mathbf{A}, \mathbf{B}\rangle = {\rm tr}\left(\mathbf{A}^T\mathbf{B}\right)$. Throughout this paper, the inequality between vectors or matrices denotes componentwise inequalities. All integrals and expectations are implicitly assumed to exist whenever they appear.

\subsection{Elliptical distributions}
The class of multivariate elliptical distributions is a natural extension to the class of multivariate normal distributions \citep[cf.][]{Fang1990}. An $n$-dimensional random vector $\boldsymbol{X}$ is said to have an elliptical distribution with location parameter $\boldsymbol{\mu}$, scale parameter $\boldsymbol{\Sigma}$ and characteristic generator $\psi$ (denoted by ${\rm ELL}_n\left( \boldsymbol{\mu},\boldsymbol{\Sigma},\psi\right)$) if its characteristic function has the form
\begin{equation} \label{ell-chara}
\Psi_X(t)={\rm exp}\left( i\boldsymbol{t}^T \boldsymbol{\mu}\right) \psi \left( \boldsymbol{t}^T \boldsymbol{\Sigma} \boldsymbol{t}\right),
\end{equation}
where $\psi$ satisfies $\psi(0) = 1$. If $\boldsymbol{X}$ has a density function,
then the density has the form
\begin{equation} \nonumber
f_{\boldsymbol{X}}(\boldsymbol{x})=\frac{c}{\sqrt{\lvert \boldsymbol{\Sigma} \rvert }} g\left( (\boldsymbol{x}-\boldsymbol{\mu})^T \boldsymbol{\Sigma}^{-1} (\boldsymbol{x}-\boldsymbol{\mu})\right),
\end{equation}
where
\begin{equation} \label{cons_cp}
c= \frac{\Gamma(n/2)}{\pi^{n/2}}\left( \int_0^\infty z^{n/2-1}g(z)dz \right)^{-1}
\end{equation}
is called the normalizing constant and $g$ is called the density generator. Note that for a given characteristic generator $\psi$, the density generator and/or the normalizing constant may depend on the dimension of the random vector $\boldsymbol{X}$. Often one considers the class of elliptical distributions of dimensions 1, 2, 3..., all derived from the same characteristic generator $\psi$. If these distributions have densities, we will denote their respective density generators and normalizing constants by $g_n$ and $c_n$, where the subscript $n$ denotes the dimension of the random vector $\boldsymbol{X}$.
One sometimes writes ${\rm ELL}_n\left( \boldsymbol{\mu},\boldsymbol{\Sigma},g_n\right)$ for the $n$-dimensional elliptical distributions generated from the function $g_n$. Some families of elliptical distributions with their density generators are presented in Table \ref{table1}.

\begin{table}
  \tbl{Some families of elliptical distributions with their density generators}
  {\begin{tabular}{ll}
    \hline
    Family& Density generator \\ \hline
    Cauchy& $g_n(u) = (1+u)^{-(n+1)/2}$\\
    Exponential power & $g_n(u) = {\rm exp}\left( -\frac{1}{s}(u)^{s/2}\right)$, $s>1$\\
    Laplace   & $g_n(u) = {\rm exp}\left( -\sqrt{u}\right)$ \\
    Normal& $g_n(u) = {\rm exp}\left( -u/2\right)$  \\
    Student   & $g_n(u) = \left(1+\frac{u}{m}\right)^{-(n+m)/2}$, $m$ is a positive integer \\
    Logistic  & $g_n(u) = {\rm exp}\left(-u\right)\left(1+{\rm exp}\left(-u\right)\right)^{-2} $ \\ \hline
    \end{tabular}}
  \label{table1}
  \end{table}

\begin{lemma} \label{id-ell}
Let $\boldsymbol{X} \sim {\rm ELL}_n\left( \boldsymbol{\mu},\boldsymbol{\Sigma},\psi\right)$, then:
\begin{enumerate}
\item The mean vector $E(\boldsymbol{X})$ (if exists) coincides with the location vector and the covariance matrix
$Cov(\boldsymbol{X})$ (if exists), being $-2\psi'(0)\boldsymbol{\Sigma}$;
\item $\boldsymbol{X}$ admits the stochastic representation
\begin{equation}  \label{rep-ell}
\boldsymbol{X} \overset{d}{=} \boldsymbol{\mu} + {R} \mathbf{A}^T \boldsymbol{U}^{(n)},
\end{equation}
where $\mathbf{A}$ is a square matrix such that $\mathbf{A}^T\mathbf{A} = \boldsymbol{\Sigma}$,
$\boldsymbol{U}^{(n)}$ is uniformly distributed on the unit sphere $S^{n-1} = \{ \boldsymbol{u} \in \mathbb{R}^n \vert \boldsymbol{u}^T\boldsymbol{u}=1 \}$,
${R} \geq 0$ is the random variable with distribution function $F$ called the generating
variable and $F$ is called the generating distribution function, ${R}$ and $\boldsymbol{U}^{(n)}$ are independent.
\item Multivariate elliptical distribution is closed under affine transformations.
Considering $\boldsymbol{Y} = \mathbf{B}\boldsymbol{X} + \boldsymbol{b}$,  where $\mathbf{B}$ is a $m \times n$
matrix with $m < n$ and $rank(\mathbf{B}) = m$ and $\boldsymbol{b} \in \mathbb{R}^m$, then
$\boldsymbol{Y} \sim {\rm ELL}_m\left( \mathbf{B} \boldsymbol{\mu}+ \boldsymbol{b},\mathbf{B}\boldsymbol{\Sigma}\mathbf{B}^T,\psi\right)$.
\end{enumerate}
\end{lemma}
\cite{yin2019stochastic} provided an important identity for multivariate elliptical distributions as follows.
\begin{lemma} \label{ID-multi}
\citep{yin2019stochastic} Let $\boldsymbol{X} \sim {\rm ELL}_n\left( \boldsymbol{\mu}^x,\boldsymbol{\Sigma}^x,\psi\right)$
and $\boldsymbol{Y} \sim {\rm ELL}_n\left( \boldsymbol{\mu}^y,\boldsymbol{\Sigma}^y,\psi\right)$
with $\boldsymbol{\Sigma}^x$ and $\boldsymbol{\Sigma}^y$ positive definite.
Let $\phi_\lambda$ be the density function of
\begin{equation} \nonumber
{\rm ELL}_n\left( \lambda\boldsymbol{\mu}^y+\left(1-\lambda\right)\boldsymbol{\mu}^x,\lambda\boldsymbol{\Sigma}^y+\left(1-\lambda\right)\boldsymbol{\Sigma}^x,\psi\right), 0 \leq \lambda \leq 1,
\end{equation}
and $\phi_{1\lambda}$ be the density function of
\begin{equation} \nonumber
{\rm ELL}_n\left( \lambda\boldsymbol{\mu}^y+\left(1-\lambda\right)\boldsymbol{\mu}^x,\lambda\boldsymbol{\Sigma}^y+\left(1-\lambda\right)\boldsymbol{\Sigma}^x,\psi_1\right), 0 \leq \lambda \leq 1,
\end{equation}
where
\begin{equation} \nonumber
\psi_1\left( u\right) = \frac{1}{E\left( {R}^2 \right)} \int_0^{+\infty} {}_0F_1 \left(\frac{n}{2}+1;-\frac{r^2u}{4}\right) r^2 \mathbb{P}\left( {R} \in dr\right).
\end{equation}
Here
\begin{equation} \nonumber
{}_0F_1 \left(\gamma;z\right) = \sum_{k=0}^{\infty} \frac{\Gamma\left(\gamma\right)}{\Gamma\left(\gamma+k\right)} \frac{z^k}{k!}
\end{equation}
is the generalized hypergeometric series of order $(0, 1)$, ${R}$ is defined by (\ref{rep-ell}) with $E\left( {R}^2\right) < \infty$ and $\mathbb{P}(A)$ means the probability of the event $A$. Moreover, assume that $f: \mathbb{R}^n \to \mathbb{R}$ is twice continuously differentiable and satisfies some polynomial growth
conditions at infinity:
\begin{equation} \nonumber
f\left(\mathbf{x}\right) = O\left( \left\|\mathbf{x}\right\| \right), \nabla f\left(\mathbf{x}\right) = O\left( \left\|\mathbf{x}\right\| \right).
\end{equation}
Then,
\begin{equation} \nonumber
\begin{split}
E\left[ f\left(\boldsymbol{Y}\right) \right] - E\left[ f\left(\boldsymbol{X}\right) \right] =&\int_0^1 \int_{\mathbb{R}^n} \left( \boldsymbol{\mu}^y - \boldsymbol{\mu}^x\right)^T \nabla f\left(\mathbf{x}\right) \phi_\lambda\left(\mathbf{x}\right) d\mathbf{x}d\lambda \\
&+\frac{E\left( {R}^2\right)}{2n} \int_0^1 \int_{\mathbb{R}^n} {\rm tr}\left( \left( \boldsymbol{\Sigma}^y -\boldsymbol{\Sigma}^x\right) H_f\left(\mathbf{x}\right)\right) \phi_{1\lambda}\left(\mathbf{x}\right)d\mathbf{x}d\lambda.
\end{split}
\end{equation}
\end{lemma}
\subsection{Integral Stochastic Orders}
Given two $n$-dimensional random vectors $\boldsymbol{X}$ and $\boldsymbol{Y}$, integral stochastic orders define orderings between $\boldsymbol{X}$ and $\boldsymbol{Y}$ by comparing $Ef\left( \boldsymbol{Y}\right)$ and $Ef\left( \boldsymbol{X}\right)$.  Let $\boldsymbol{F}$ be a class of measurable functions $f: \mathbb{R}^{n} \to \mathbb{R}$. Then, we say that $\boldsymbol{X} \le_F \boldsymbol{Y}$ if $Ef\left( \boldsymbol{X} \right) \le Ef\left( \boldsymbol{Y} \right)$ holds for all $f \in \boldsymbol{F}$. A general study on integral stochastic orders has been given by \cite{muller1997stochastic}.
\begin{definition}
For any function $f: \mathbb{R}^{p} \to \mathbb{R}$, the difference operator $\boldsymbol{\Delta}_i^\epsilon$, $1\leq i \leq p$, $\epsilon > 0$ is defined as $\boldsymbol{\Delta}_i^\epsilon f(\mathbf{x}) = f(\mathbf{x}+ \epsilon \mathbf{e}_i) - f(\mathbf{x})$, where $\mathbf{e}_i$ stands for the $i$-th unit basis vector of $\mathbb{R}^n$. Then
\begin{enumerate}
\item $f$ is supermodular if $\boldsymbol{\Delta}_i^{\epsilon_1}\boldsymbol{\Delta}_j^{\epsilon_2}f(\mathbf{x}) \geq 0$ holds for all $\mathbf{x} \in \mathbb{R}^n$, $\epsilon_1, \epsilon_2\geq 0$ and $1 \leq i < j \leq n$;
\item $f$ is directionally convex if $\boldsymbol{\Delta}_i^{\epsilon_1}\boldsymbol{\Delta}_j^{\epsilon_2}f(\mathbf{x}) \geq 0$ holds for all $\mathbf{x} \in \mathbb{R}^n$, $\epsilon_1, \epsilon_2\geq 0$ and $1 \leq i, j \leq n$;
\item $f$ is $\boldsymbol{\Delta}$-monotone if $\boldsymbol{\Delta}_{i_1}^{\epsilon_1}\boldsymbol{\Delta}_{i_2}^{\epsilon_2} \dots {\Delta}_{i_k}^{\epsilon_k}f(\mathbf{x}) \geq 0$ holds for all $\mathbf{x} \in \mathbb{R}^n$, $\epsilon_i\geq 0$ for $1 \geq i \geq k$ and for any subset $\{i_1,i_2, \dots,i_k\} \subseteq \{1,2,\dots,n\}$.
\end{enumerate}
\end{definition}
\begin{remark}
These three classes of functions can be characterized by their derivitives:
\begin{enumerate}
\item $f$ is supermodular if and only if $\frac{\partial^2}{\partial x_i \partial x_j} f(\mathbf{x}) \geq 0$ holds for all $\mathbf{x} \in \mathbb{R}^n$ and $1 \leq i < j \leq n$;
\item $f$ is directionally convex if and only if $\frac{\partial^2}{\partial x_i \partial x_j} f(\mathbf{x}) \geq 0$ holds for all $\mathbf{x} \in \mathbb{R}^n$ and $1 \leq i, j \leq n$.
\item $f$ is $\boldsymbol{\Delta}$-monotone if and only if $\frac{\partial^k}{\partial x_{i_1} \dots \partial x_{i_k}}f(\mathbf{x}) \geq 0$ holds for all $\mathbf{x} \in \mathbb{R}^n$, $1 \leq k \leq n$ and $1 \leq i_1 < \dots < i_k \leq n$.
\end{enumerate}
\end{remark}
\begin{definition}
\citep{arlotto2009hessian} An $n \times n$ matrix $\mathbf{A}$ is called copositive if the quadratic form $\mathbf{x^TAx} \geq 0$ for all $\mathbf{x} \geq 0$, and $\mathbf{A}$ is called completely positive if there exists a nonnegative $m \times n$ matrix $\mathbf{B}$ such that $\mathbf{A} = \mathbf{B}^T\mathbf{B}$.
\end{definition}
We use $\mathcal{C}_{cop}$ to denote the cone of copositive matrices and $\mathcal{C}_{cp}$ to denote the cone of completely positive matrices. We use $\mathcal{C}^*$ to denote the dual of the closed convex cone $\mathcal{C}$, i.e. $\mathcal{C}^* = \{ \mathbf{B}: {\rm tr}(\mathbf{A}^T\mathbf{B}) \geq 0 , \forall \mathbf{A} \in \mathcal{C} \}$. \cite{arlotto2009hessian} proved that the cones of $\mathcal{C}_{cop}$ and $\mathcal{C}_{cp}$ are both closed and convex, and
\begin{equation} \label{rela-cpcop}
\mathcal{C}_{cop}^* = \mathcal{C}_{cp}, ~~\mathcal{C}_{cp}^* =\mathcal{C}_{cop}.
\end{equation}

\begin{definition} \label{def-order}
We say $\boldsymbol{X}$ is smaller than $\boldsymbol{Y}$ in the:
\begin{enumerate}
\item Usual stochastic order, i.e. $\boldsymbol{X} \le_{st} \boldsymbol{Y}$, if $Ef\left(\boldsymbol{X}\right) \le Ef\left(\boldsymbol{Y}\right)$ for all increasing functions;
\item Positive linear usual stochastic order, i.e. $\boldsymbol{X} \le_{plst} \boldsymbol{Y}$, if $ \langle \boldsymbol{a}, \boldsymbol{X} \rangle \le_{st} \langle \boldsymbol{a}, \boldsymbol{Y} \rangle$ for all $\boldsymbol{a} \in \mathbb{R}_+^n$;
\item Convex order, i.e. $\boldsymbol{X} \le_{cx} \boldsymbol{Y}$, if $Ef\left(\boldsymbol{X}\right) \le Ef\left(\boldsymbol{Y}\right)$ for all convex functions;
\item Linear convex order, i.e. $\boldsymbol{X} \le_{lcx} \boldsymbol{Y}$, if $\langle \boldsymbol{a}, \boldsymbol{X} \rangle \le_{cx} \langle \boldsymbol{a}, \boldsymbol{Y} \rangle$ for all $\boldsymbol{a} \in \mathbb{R}^n$;
\item Increasing convex order (stop-loss order), i.e. $\boldsymbol{X} \le_{icx}(\le_{sl}) \boldsymbol{Y}$, if $Ef\left(\boldsymbol{X}\right) \le Ef\left(\boldsymbol{Y}\right)$ for all increasing convex functions;
\item Increasing linear convex order, i.e. $\boldsymbol{X} \leq_{ilcx} \boldsymbol{Y}$, if $\langle \boldsymbol{a}, \boldsymbol{X} \rangle \le_{icx} \langle \boldsymbol{a}, \boldsymbol{Y} \rangle$ for all $\boldsymbol{a} \in \mathbb{R}^n$;
\item Increasing positive linear convex order, i.e. $\boldsymbol{X} \le_{iplcx} \boldsymbol{Y}$, if $\langle \boldsymbol{a}, \boldsymbol{X} \rangle \le_{icx} \langle \boldsymbol{a}, \boldsymbol{Y} \rangle$ for all $\boldsymbol{a} \in \mathbb{R}_+^n$;
\item Directionally convex order, i.e. $\boldsymbol{X} \le_{dcx} \boldsymbol{Y}$, if $Ef\left(\boldsymbol{X}\right) \le Ef\left(\boldsymbol{Y}\right)$ for all directionally convex functions;
\item Componentwise convex order, i.e. $\boldsymbol{X} \le_{ccx} \boldsymbol{Y}$, if $Ef\left(\boldsymbol{X}\right) \le Ef\left(\boldsymbol{Y}\right)$ for all componentwise convex functions;
\item Upper orthant order, i.e. $\boldsymbol{X} \le_{uo} \boldsymbol{Y}$, if $Ef\left(\boldsymbol{X}\right) \le Ef\left(\boldsymbol{Y}\right)$ for all $\boldsymbol{\Delta}$-monotone functions;
\item Supermodular order, i.e. $\boldsymbol{X} \le_{sm} \boldsymbol{Y}$, if $Ef\left(\boldsymbol{X}\right) \le Ef\left(\boldsymbol{Y}\right)$ for all supermodular functions;
\item Copositive order, i.e. $\boldsymbol{X} \le_{cp} \boldsymbol{Y}$, if $Ef\left(\boldsymbol{X}\right) \le Ef\left(\boldsymbol{Y}\right)$ for all twice differentiable functions such that $\mathbf{H}_f(\mathbf{x}) \in \mathcal{C}_{cp}$;
\item Completely positive order, i.e. $\boldsymbol{X} \le_{cop} \boldsymbol{Y}$, if $Ef\left(\boldsymbol{X}\right) \le Ef\left(\boldsymbol{Y}\right)$ for all twice differentiable functions such that $\mathbf{H}_f(\mathbf{x}) \in \mathcal{C}_{cop}$.
\end{enumerate}
\end{definition}
\citet{denuit2002smooth} points out that many integral stochastic orders, including the first ten orders in Definition \ref{def-order}, have a generator consisting of infinitely differentiable functions. Taking the usual stochastic order as an example, if $Ef\left(\boldsymbol{X}\right) \le Ef\left(\boldsymbol{Y}\right)$ for all infinitely differentiable increasing functions $f$, it is sufficient to say $\boldsymbol{X} \leq_{st} \boldsymbol{Y}$.

The existence of the following chain of implications among the aforementioned stochastic orderings is well known \citep[c.f.][]{Scarsini1998,pan2016stochastic,Amiri2020Linear}:
\begin{equation} \label{im-chain}
\begin{split}
\boldsymbol{X} &\le_{cx} \boldsymbol{Y} \Rightarrow \boldsymbol{X} \le_{lcx} \boldsymbol{Y} \Leftrightarrow \boldsymbol{X} \le_{ilcx} \boldsymbol{Y} \\
& \Downarrow \\
\boldsymbol{X} \le_{plst} \boldsymbol{Y} \Leftarrow \boldsymbol{X} \le_{st} \boldsymbol{Y} \Rightarrow \boldsymbol{X} &\le_{icx} \boldsymbol{Y} \Rightarrow \boldsymbol{X} \le_{iplcx} \boldsymbol{Y} .
\end{split}
\end{equation}


\section{Generalized Location-Scale Mixture of Elliptical Distributions} \label{lse}
\par Mixtures of distributions occur frequently both in theory and applications of probability and statistics. For example, in the simplest case it may be reasonable to assume that one is dealing with the given proportion of normal populations with different means and/or variances. Mixtures of distributions can be used as a method to describe how external factors, which may not exert their influence on samples equally, influence the original distribution.
\par Consider the $n$-dimensional random vector $\boldsymbol{Y}$ that can be expressed as
\begin{equation} \label{lse-rep}
\boldsymbol{Y} \overset{d}{=} \boldsymbol{\mu} + \alpha(\boldsymbol{Z}) \boldsymbol{X} + \beta(\boldsymbol{Z})\boldsymbol{\delta},
\end{equation}
where $\boldsymbol{\mu},\boldsymbol{\delta} \in \mathbb{R}^n$, $\alpha,\beta: \mathbb{R}^q \to \mathbb{R}_+$, $\boldsymbol{X} \sim {\rm ELL}_n\left(\boldsymbol{0},\boldsymbol{\Sigma},\psi\right)$ with a positive definite matrix $\boldsymbol{\Sigma}$ and it has density generator $g_n$, and $\boldsymbol{Z}$ is a $q$-dimensional random vector with CDF $H^q(\mathbf{z})$ and independent of $\boldsymbol{X}$. Then, the random vector $\boldsymbol{Y}$ is said to have a generalized location-scale mixture of elliptical distributions, which will be denoted by ${\rm GLSE}_n(\boldsymbol{\mu},\boldsymbol{\Sigma},\boldsymbol{\delta},\psi,\alpha,\beta,H^q)$. Here, $\boldsymbol{\mu}$, $\boldsymbol{\Sigma}$ and $\boldsymbol{\delta}$ are the vectors of location parameters, scale parameters and skewness parameters of this distribution, respectively. The conditional representation of $\boldsymbol{Y}$ can be expressed as
\begin{equation} \label{condi-rep}
\boldsymbol{Y} | \boldsymbol{Z} \sim {\rm ELL}_n\left( \boldsymbol{\mu} + \beta(\boldsymbol{Z})\boldsymbol{\delta}, \alpha^2(\boldsymbol{Z}) \boldsymbol{\Sigma},\psi \right).
\end{equation}
Therefore, the density and characteristic functions of $\boldsymbol{Y}$ take the forms
\begin{equation} \label{lse-pdf}
f(\mathbf{y}) = \int_{\mathbb{R}^q} \frac{c_n}{\alpha(\mathbf{z}) \sqrt{|\boldsymbol{\Sigma}|}} g_n\left(
\frac{1}{\alpha^2(\mathbf{z})}(\mathbf{y}-\boldsymbol{\mu} - \beta(\mathbf{z})\boldsymbol{\delta})^T \boldsymbol{\Sigma}^{-1}(\mathbf{y}-\boldsymbol{\mu} - \beta(\mathbf{z})\boldsymbol{\delta})\right) dH^q(\mathbf{z}),
\end{equation}
where $c_n$ follows (\ref{cons_cp}) and
\begin{equation} \nonumber 
\Psi(\mathbf{t}) = {\rm exp}(i\mathbf{t}^T \boldsymbol{\mu}) E_{\boldsymbol{Z}}\left( {\rm exp}\left( i \beta(\boldsymbol{Z}) \mathbf{t}^T \boldsymbol{\delta} \right) \psi\left( \alpha^2(\boldsymbol{Z}) \mathbf{t}^T \boldsymbol{\Sigma}\mathbf{t}\right)\right).
\end{equation}
Provided that $E\left( \alpha^2(\boldsymbol{Z})\right)$, $E\left(\beta(\boldsymbol{Z})\right)$, $Var\left(\beta(\boldsymbol{Z})\right)$, the mean vector and the covariance matrix of $\boldsymbol{Y}$ exist, then the mean vector and the covariance matrix of $\boldsymbol{Y}$ are given by
\begin{equation} \nonumber 
E\left(\boldsymbol{Y}\right) = \boldsymbol{\mu} + E\left(\beta(\boldsymbol{Z})\right)\boldsymbol{\delta},
\end{equation}
and
\begin{equation} \label{lse-cov}
Cov(\boldsymbol{Y}) = -2\psi'(0) E\left( \alpha^2(\boldsymbol{Z})\right) \boldsymbol{\Sigma} + Var\left(\beta(\boldsymbol{Z})\right) \boldsymbol{\delta}\boldsymbol{\delta}^T.
\end{equation}

The following lemma shows that the GLSE distribution is closed under affine transformations.
\begin{lemma} \label{lse-aff}
Let $\boldsymbol{Y} \sim {\rm GLSE}_n(\boldsymbol{\mu},\boldsymbol{\Sigma},\boldsymbol{\delta},\psi,\alpha,\beta,H)$, and $\mathbf{B}$ be a $m \times n$ matrix with $m < n$ and $rank(\mathbf{B}) = m$ and $\boldsymbol{b} \in \mathbb{R}^m$, then $\mathbf{B}\boldsymbol{Y}+\boldsymbol{b} \sim {\rm GLSE}_m(\mathbf{B}\boldsymbol{\mu}+\boldsymbol{b},\mathbf{B}\boldsymbol{\Sigma}\mathbf{B}^T,\mathbf{B}\boldsymbol{\delta},\psi,\alpha,\beta,H)$.
\end{lemma}
\begin{proof}
The characteristic function of $\mathbf{B}\boldsymbol{Y}+\boldsymbol{b}$ is obtained as
\begin{equation} \nonumber
  \Psi_{\mathbf{B}\boldsymbol{Y}+\boldsymbol{b}}(\mathbf{t}) = {\rm exp}(i\mathbf{t}^T \mathbf{B}\boldsymbol{\mu}) {\rm exp}(i\mathbf{t}^T \boldsymbol{b}) E_Z\left( {\rm exp}\left( i \beta(\boldsymbol{Z}) \mathbf{t}^T \mathbf{B}\boldsymbol{\delta} \right) \psi\left( \alpha^2(\boldsymbol{Z}) \mathbf{t}^T \mathbf{B}\boldsymbol{\Sigma}\mathbf{B}^T\mathbf{t}\right)\right),
\end{equation}
which shows the result.
\end{proof}
The following theorem illustrates a peculiar property for GLSE distribution, that is, multivariate GLSE distributions are closed under skewness-parallel location-scale mixture.
\begin{theorem} \label{mix-close}
Let $\boldsymbol{Y} \sim {\rm GLSE}_n(\boldsymbol{0},\boldsymbol{\Sigma},\boldsymbol{\delta},\psi,\alpha_1,\beta_1,H_1)$ and $\boldsymbol{Z}_1$ be a $q$-dimensional random vector with CDF $H_1(\mathbf{z})$. Consider the location-scale mixture of $\boldsymbol{Y}$, i.e.
\begin{equation} \label{mixY-rp}
\hat{\boldsymbol{Y}} \overset{d}{=} \boldsymbol{\mu} + \alpha_2(\boldsymbol{Z}_2) \boldsymbol{Y} + \beta_2(\boldsymbol{Z}_2) \boldsymbol{\delta}_2,
\end{equation}
where $\boldsymbol{\mu} \in \mathbb{R}^n$, $\boldsymbol{\delta}_2 \in \mathbb{R}^n$ and there exists $k \in \mathbb{R}$ such that $\boldsymbol{\delta}_2 = k\boldsymbol{\delta}$, $\alpha_2,\beta_2: \mathbb{R}^p \to \mathbb{R}_+$ and $\boldsymbol{Z}_2$ is a $p$-dimensional random vector with CDF $H^p_2(\mathbf{z})$ and independent of $\boldsymbol{Y}$. Set $\boldsymbol{Z}^* = (\boldsymbol{Z}_1^T,\boldsymbol{Z}_2^T)^T$, $H^{p+q}_*(\mathbf{z})$ be the CDF of $\boldsymbol{Z}^*$, $\alpha^*(\boldsymbol{Z}^*) = \alpha_1(\boldsymbol{Z}_1)\alpha_2(\boldsymbol{Z}_2)$ and $\beta^*(\boldsymbol{Z}^*)= \beta_1(\boldsymbol{Z}_1)\alpha_2(\boldsymbol{Z}_2)+ k\beta_2(\boldsymbol{Z}_2)$. Then $\hat{\boldsymbol{Y}} \sim {\rm GLSE}_n (\boldsymbol{\mu},\boldsymbol{\Sigma},\boldsymbol{\delta},\psi,\alpha^*,\beta^*,H_*)$.
\end{theorem}
\begin{proof}
The characteristic function of $\hat{\boldsymbol{Y}}$ can be derived immediately from stochastic representation (\ref{mixY-rp}):
\begin{equation*}
\Psi_{\hat{\boldsymbol{Y}}} = {\rm exp}(i\mathbf{t}^T \boldsymbol{\mu}) E_{\boldsymbol{Z}^*}\left[ {\rm exp}\left( i (\beta_1(\boldsymbol{Z}_1)\alpha_2(\boldsymbol{Z}_2)+ k\beta_2(\boldsymbol{Z}_2)) \mathbf{t}^T \boldsymbol{\delta} \right) \psi\left( \alpha_1^2(\boldsymbol{Z}_1)\alpha_2^2(\boldsymbol{Z}_2) \mathbf{t}^T \boldsymbol{\Sigma}\mathbf{t}\right)\right],
\end{equation*}
then the desired result can be obtained.
\end{proof}
The one-dimensional case of Theorem \ref{mix-close} can be established in the following corollary that univariate GLSE distributions are closed under location-scale mixture.
\begin{corollary} \label{coll-mix}
Let $Y \sim {\rm GLSE}_1(0,\sigma,\delta,\psi,\alpha_1,\beta_1,H)$. Consider the location-scale mixture of $\boldsymbol{Y}$, i.e.
\begin{equation*}
\hat{Y} \overset{d}{=} {\mu} + \alpha_2(\boldsymbol{Z}_2) {Y} + \beta_2(\boldsymbol{Z}_2) \delta_2,
\end{equation*}
where ${\mu},{\delta}_2 \in \mathbb{R}$, $\alpha_2,\beta_2: \mathbb{R}^q \to \mathbb{R}_+$ and $\boldsymbol{Z}_2$ is a $q$-dimensional random vector with CDF $H_2(\mathbf{z})$ and independent of $\boldsymbol{Y}$. Set $\boldsymbol{Z}^* = (\boldsymbol{Z}_1^T,\boldsymbol{Z}_2^T)^T$, $H^*(\mathbf{z})$ be the CDF of $\boldsymbol{Z}^*$, $\alpha^*(\boldsymbol{Z}^*) = \alpha_1(\boldsymbol{Z}_1)\alpha_2(\boldsymbol{Z}_2)$, $\beta^*(\boldsymbol{Z}^*)= \beta_1(\boldsymbol{Z}_1)\alpha_2(\boldsymbol{Z}_2)+ k\beta_2(\boldsymbol{Z}_2)$ and $k=\delta_2/\delta$. Then $\hat{\boldsymbol{Y}} \sim {\rm GLSE}_n ({\mu},{\sigma},{\delta},\psi,\alpha^*,\beta^*,H^*)$.
\end{corollary}
The family of GLSE distributions is large enough to contain several subfamilies of symmetric and non-symmetric distributions. A considerable amount of well-known distributions can be seen as special cases of GLSE distributions, and we introduce some of them here.
\begin{enumerate}
\item Skew-normal distributions. Follow the notations in (\ref{SN-rep}), and let $\hat{\boldsymbol{\delta}} = \boldsymbol{\sigma \delta}$, then (\ref{SN-rep}) can be rewritten as
\begin{equation} \label{SN-rewrite}
\boldsymbol{Y} \overset{d}{=} \boldsymbol{\mu} + \hat{\boldsymbol{X}} + \boldsymbol{Z} \hat{\boldsymbol{\delta}},
\end{equation}
where $\hat{\boldsymbol{X}} \sim N_n(0,{\boldsymbol{\Sigma}} - \hat{\boldsymbol{\delta}}\hat{\boldsymbol{\delta}}^T)$. Once we set $q = 1$, $\alpha(z) = 1$, $\beta(z) = z$ in (\ref{lse-rep}), then the stochastic representation (\ref{SN-rewrite}) can be obtained. Furthermore, in light of Corollary \ref{coll-mix}, it can be claimed that the variance-mean mixture of the univariate skew normal distribution introduced by \cite{Arslan2015} can be seen special case of the GLSE distribution.
\item  Location-scale mixture of elliptical distributions \citep{zuo2021}: $q = 1$, $\alpha(z) = \sqrt{z}$, $\beta(z) = z$. It becomes mean-variance mixture of multinormal distributions  \citep{Barndorff1982Normal} when setting characteristic generator $\psi(u) = {\rm exp}\left( -u/2\right)$.
\item  Generalized hyperbolic skew-slash distribution \citep{Arslan2008}: $\psi(u) = {\rm exp}\left( -u/2\right)$, $q = 1$, $\boldsymbol{Z} \sim beta(\lambda,1)$, $\alpha(z) = z^{-\frac{1}{2}}$, $\beta(z) = z^{-1}$. This distribution will be denoted by ${\rm GHSS}_n(\boldsymbol{\mu},\boldsymbol{\Sigma},\boldsymbol{\delta},\lambda)$.
\item Generalized Hyperbolic distribution \citep{Barndorff1981}: $\psi(u) = {\rm exp}\left( -u/2\right)$, $q = 1$, $\alpha(z) = \sqrt{z}$, $\beta(z) = z$ and $\boldsymbol{Z}$ follows Generalized inverse Gussian distribution with density
\begin{equation} \nonumber
h(\omega) = \frac{(\tau / \chi)^{\frac{\lambda}{2}}}{2K_\lambda(\sqrt{\chi \tau })} \omega^{\lambda -1} {\rm exp}\left(- \frac{1}{2} \left(\frac{\chi}{\omega} + \tau \omega\right) \right), \omega > 0,
\end{equation}
where parameters follow
\begin{equation} \nonumber
    \left\{
    \begin{array}{lr}
    \chi > 0 \quad and \quad \tau \geq 0, & if \quad \lambda<0,  \\
    \chi > 0 \quad and \quad \tau > 0, & if \quad \lambda = 0, \\
    \chi \geq 0 \quad and \quad \tau > 0, & if \quad \lambda > 0
    \end{array}
    \right.
    \end{equation}
and $K_\lambda$ being the Bessel function of the third kind with index $\lambda$.
\item The model of nonsymmetric security returns (\cite{simaan1993}, defined by (\ref{fina-rep}) in this paper): $\alpha(z) = 1$, $\beta(z) = z$ and $\boldsymbol{Z}$ be a univariate random variable with any non-elliptical distribution.
\end{enumerate}
\par The following lemma provides an identity for GLSE distributions, which provides us an efficient way to prove some stochastic ordering results.
\begin{lemma} \label{id-lse}
Assume $\boldsymbol{Y}_1 \sim {\rm GLSE}_n\left(\boldsymbol{\mu}_1, \boldsymbol{\Sigma}_1, \boldsymbol{\delta}_1, \psi, \alpha, \beta, H\right)$ and $\boldsymbol{Y}_2 \sim {\rm GLSE}_n\left(\boldsymbol{\mu}_2, \boldsymbol{\Sigma}_2, \boldsymbol{\delta}_2, \psi, \alpha, \beta, H\right)$. If all the conditions in Lemma \ref{id-ell} are satisfied, then
\begin{equation} \label{GLSE-id}
\begin{split}
E\left[ f\left(\boldsymbol{Y}_1\right) \right] - E\left[ f\left(\boldsymbol{Y}_2\right) \right] =&\int_{\mathbb{R}^q} \int_0^1 \int_{\mathbb{R}^n} \left( \boldsymbol{\mu}_2 - \boldsymbol{\mu}_1 + \beta(\mathbf{z}) \left(\boldsymbol{\delta}_2 - \boldsymbol{\delta}_1 \right)\right)^T \nabla f\left(\mathbf{x}\right) \phi_\lambda\left(\mathbf{x}\right) d\mathbf{x}d\lambda dH^q(\mathbf{z}) \\
&+\frac{E\left( {R}^2\right)}{2n} \int_{\mathbb{R}^q} \int_0^1 \int_{\mathbb{R}^n} \alpha^2(\mathbf{z}) {\rm tr}\left( \left( \boldsymbol{\Sigma}_2 -\boldsymbol{\Sigma}_1\right) H_f\left(\mathbf{x}\right)\right) \phi_{1\lambda}\left(\mathbf{x}\right)d\mathbf{x}d\lambda dH^q(\mathbf{z}).
\end{split}
\end{equation}
\end{lemma}
\begin{proof}
Applying double expectation formula to the left-hand side of (\ref{GLSE-id}), we have
\begin{equation} \nonumber
E\left[ f\left(\boldsymbol{Y}_1\right) \right] - E\left[ f\left(\boldsymbol{Y}_2\right) \right]  = \int_{\mathbb{R}^q} E\left[ f\left(\boldsymbol{Y}_1\right) | \boldsymbol{Z} = \mathbf{z} \right] - E\left[ f\left(\boldsymbol{Y}_2\right) | \boldsymbol{Z} = \mathbf{z} \right] dH^q(\mathbf{z}).
\end{equation}
Then the desired result can be readily obtained by applying (\ref{condi-rep}) and Lemma \ref{ID-multi}.
\end{proof}
\par If one set $\boldsymbol{\delta} = 0$ in (\ref{lse-rep}), then the part of location mixture of GLSE distribution vanishes and it degenerates to scale mixture of elliptical distributions. In other words, the family of scale mixture of elliptical distributions is set up by stochastic representation $\boldsymbol{Y} \overset{d}{=} \boldsymbol{\mu} + \alpha(\boldsymbol{Z}) \boldsymbol{X}$, where the parameters are set in parallel with (\ref{lse-rep}). The randon vector $\boldsymbol{Y}$ will be denoted by ${\rm SME}(\boldsymbol{\mu},\boldsymbol{\Sigma},\psi,\alpha,H)$. Obviously, the identities presented in this section are still valid in SME case. Setting $\boldsymbol{\delta} = 0$, the PDF in (\ref{lse-pdf}) has the form
\begin{equation} \nonumber
f(\mathbf{y}) = \frac{c_n}{\alpha(\mathbf{z}) \sqrt{|\boldsymbol{\Sigma}|}} \overline{g}_n\left( (\mathbf{y} - \boldsymbol{\mu})^T \boldsymbol{\Sigma}^{-1} (\mathbf{y} - \boldsymbol{\mu})\right),
\end{equation}
where $\overline{g}_n(\mathbf{y}) = \int_{\mathbb{R}^q} g_n\left( \frac{1}{\alpha^2(z)}(\mathbf{y} - \boldsymbol{\mu})^T \boldsymbol{\Sigma}^{-1} (\mathbf{y} - \boldsymbol{\mu})\right) dH^q(\mathbf{z})$. It can be observed that the characteristic function of SME distributions is of the form (\ref{ell-chara}).
\section{Stochastic Ordering Results} \label{results}
In some cases, the density generators are arbitrarily chosen and thus too general to study the properties of GLSE distributions. We have to narrow down the variety of density generator $g$ under some specific situations. To this end, the following technical assumptions are necessarily needed.
\begin{assumption} \label{ass1}
Let $t_i = {t}/{\sigma_i}$, $\sigma_i > 0$ for $i =1,2$. We assume density generator $g$ satisfies for $\sigma_1 \neq \sigma_2$,
\begin{equation} \nonumber
\lim_{t \to +\infty} \frac{\sigma_1}{\sigma_2} \frac{g(t_2^2)}{g(t_1^2)} = \lim_{t \to -\infty} \frac{\sigma_1}{\sigma_2} \frac{g(t_2^2)}{g(t_1^2)} = C,
\end{equation}
where $C \in \overline{\mathbb{R}_+} \setminus \{ 1 \}$.
\end{assumption}
\begin{assumption} \label{ass2}
Let $t_i = {t}/{\sigma_i}$, $\sigma_i > 0$ for $i =1,2$.  We assume density generator $g$ satisfies for $\sigma_1 > \sigma_2$,
\begin{equation} \nonumber
\lim_{t \to +\infty} \frac{\sigma_1}{\sigma_2} \frac{g(t_2^2)}{g(t_1^2)} = \lim_{t \to -\infty} \frac{\sigma_1}{\sigma_2} \frac{g(t_2^2)}{g(t_1^2)} = C',
\end{equation}
where $C' \in [0,1)$.
\end{assumption}
These two technical assumptions are used to control tail behaviors of the density functions, based on which we can compare the GLSE variables by considering the limits of quotient of density functions at infinity in the sequel discussions. In Assumption \ref{ass1}, we suppose the tail behavior of the density functions are not identical when $\sigma_1$ and $\sigma_2$ are not equal. In Assumption \ref{ass2}, we suppose the limits of quotient of density functions at infinity is less than $1$ in order to ensure that $g(t^2/\sigma_2^2)/{\sigma_2} \leq g(t^2/\sigma_1^2)/{\sigma_1}$ as $t$ goes to infinity when $\sigma_1 > \sigma_2$.
\begin{remark} \label{ass-satis}
Note that Assumptions \ref{ass1} and \ref{ass2} are not strict and all the density generators presented in Table \ref{table1} follow Assumptions \ref{ass1} and \ref{ass2}. We will prove this statement in Appendix. This fact is needed for characterizing the limit behavior of the density generator.
\end{remark}
Based on the foregoing assumptions, the conditions for the usual stochastic order of univariate GLSE distribution can be established as follows.
\begin{lemma} \label{st-dim1}
Assume
\begin{equation} \label{GLSEdim1}
  \begin{split}
    Y_1 \sim & {\rm GLSE}_1\left(\mu_1, \sigma_1, \delta_1, \psi, \alpha, \beta, H\right),\\
    {Y}_2 \sim & {\rm GLSE}_1\left(\mu_2, \sigma_2, \delta_2, \psi, \alpha, \beta, H\right).
  \end{split}
  \end{equation}
\begin{enumerate}
\item If $\mu_2 - \mu_1 + \beta(\mathbf{z}) \left({\delta}_2 - {\delta}_1 \right) \geq 0$ for all $\mathbf{z}$ and ${\sigma}_1 = {\sigma}_2$, then ${Y}_1 \leq_{st} {Y}_2$.
\item If ${Y}_1 \leq_{st} {Y}_2$ and density generator $g_1$ satisfies Assumption \ref{ass1}, then $\mu_1+E\left(\beta(\mathbf{Z})\right){\delta}_1 \leq \mu_2+E\left(\beta(\mathbf{Z})\right){\delta}_2$ and ${\sigma}_1 = {\sigma}_2$.
\end{enumerate}
\end{lemma}
\begin{proof}
1. The implication follows Lemma \ref{id-lse}.
\par 2. If ${Y}_1 \leq_{st} {Y}_2$, then $E{Y}_1 \leq E{Y}_2$, obviously we have $\mu_1+E\left(\beta(\mathbf{Z})\right){\delta}_1 \leq \mu_2+E\left(\beta(\mathbf{Z})\right){\delta}_2$.
\par We claim ${\sigma}_1 = {\sigma}_2$. If ${\sigma}_1 \neq {\sigma}_2$, according to Assumption \ref{ass1}, we have
\begin{equation} \nonumber
\lim_{y \to \pm\infty} \frac{p_2(y,\mathbf{z})}{p_1(y,\mathbf{z})} = C,
\end{equation}
where
\begin{equation}  \nonumber
p_i(y,\mathbf{z}) = \frac{c_n}{\alpha(\mathbf{z}) \sigma_i} g_n\left(
  \frac{1}{\alpha^2(\mathbf{z})\sigma^2_i}(y-{\mu_i} - \beta(\mathbf{z}){\delta_i})^2\right),
\end{equation}
  and $C \in \overline{\mathbb{R}^+} \setminus \{ 1 \}$. If $C \in [0,1)$, then for sufficiently large positive $t$, $p_2(y,\mathbf{z}) < p_1(y,\mathbf{z})$, where $y>t$. Considering the CDF of ${Y}_1$ and ${Y}_2$, we have
\begin{equation} \nonumber
\overline{F}_2(t) = \int_{\mathbb{R}^q} \int_t^{+\infty} p_2(x,\mathbf{z}) dxdH^q(\mathbf{z}) < \int_{\mathbb{R}^q} \int_t^{+\infty} p_1(x,\mathbf{z}) dxdH^q(\mathbf{z}) = \overline{F}_1(t),
\end{equation}
which contradicts with ${Y}_1 \leq_{st} {Y}_2$. In parallel, if $C \in (1,+\infty]$, then for negative $t$ with sufficiently large $|t|$, $p_2(y,\mathbf{z}) > p_1(y,\mathbf{z})$, where $y<t$. So
\begin{equation}\nonumber
{F}_2(t) = \int_{\mathbb{R}^q} \int^t_{-\infty} p_2(x,\mathbf{z}) dxdH^q(\mathbf{z}) > \int_{\mathbb{R}^q} \int^t_{-\infty} p_1(x,\mathbf{z}) dxdH^q(\mathbf{z}) = {F}_1(t),
\end{equation}
which leads a contradiction to ${Y}_1 \leq_{st} {Y}_2$. Hence, we conclude ${\sigma}_1 = {\sigma}_2$.
\end{proof}
\begin{example}We consider the special case of univariate GLSE distribution to illustrate the results in Lemma \ref{st-dim1}. The survival functions of the GHSS distributions under three parameter setting are plotted in Figure \ref{fig:survival}(a). It is easy to check that the conditions in Lemma \ref{st-dim1}(1) are satisfied, and thus agrees with the usual stochastic ordering among these three distributions displayed in the plot. Moreover, Figure \ref{fig:survival}(b) provides a counterexample showing that the survival functions of these distributions cross with each other if the scale parameters of the three distributions are not identical.
\end{example}

\begin{figure}
  \centering
  \subfloat[]{%
  \resizebox*{7cm}{!}{\includegraphics{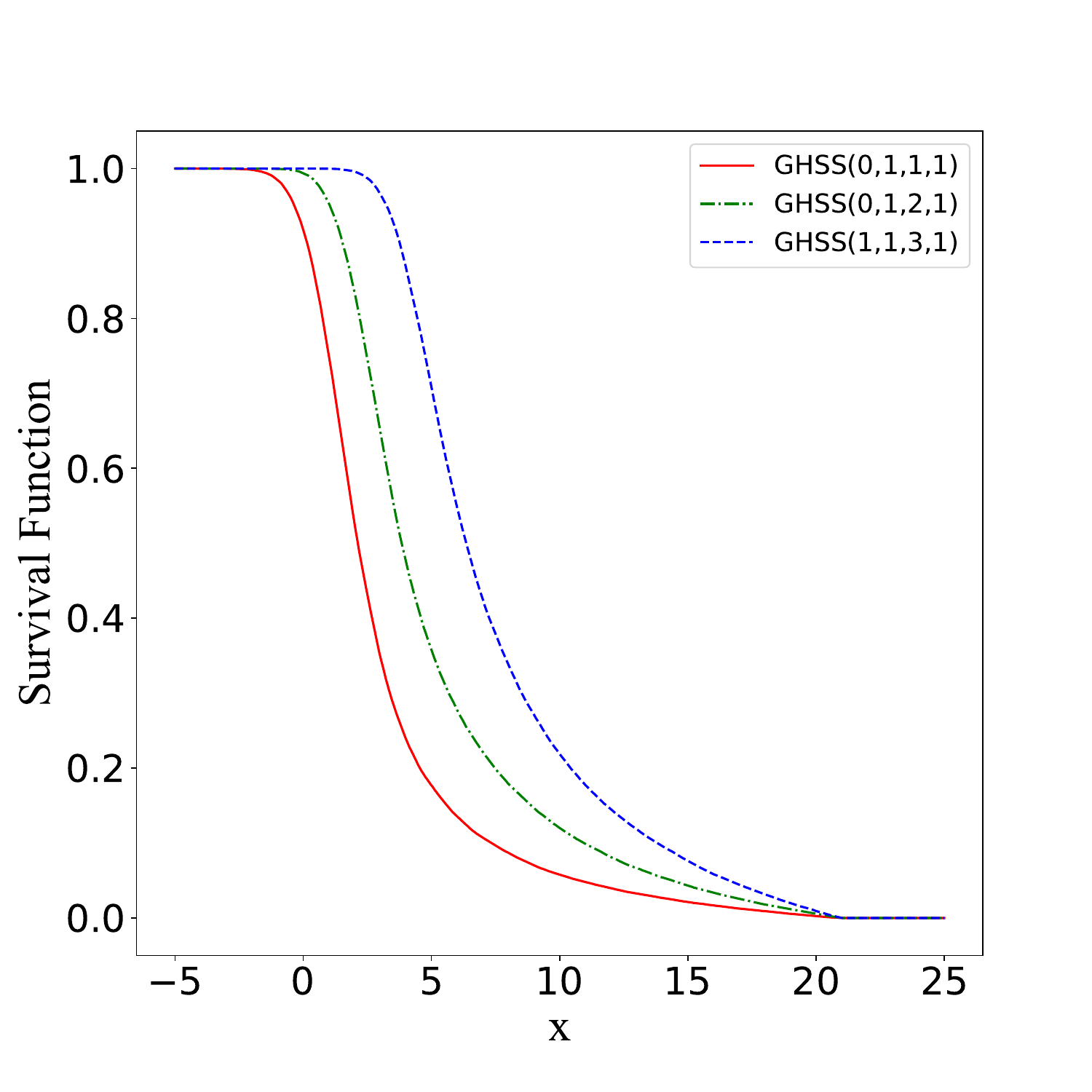}}}\hspace{5pt}
  \subfloat[]{%
  \resizebox*{7cm}{!}{\includegraphics{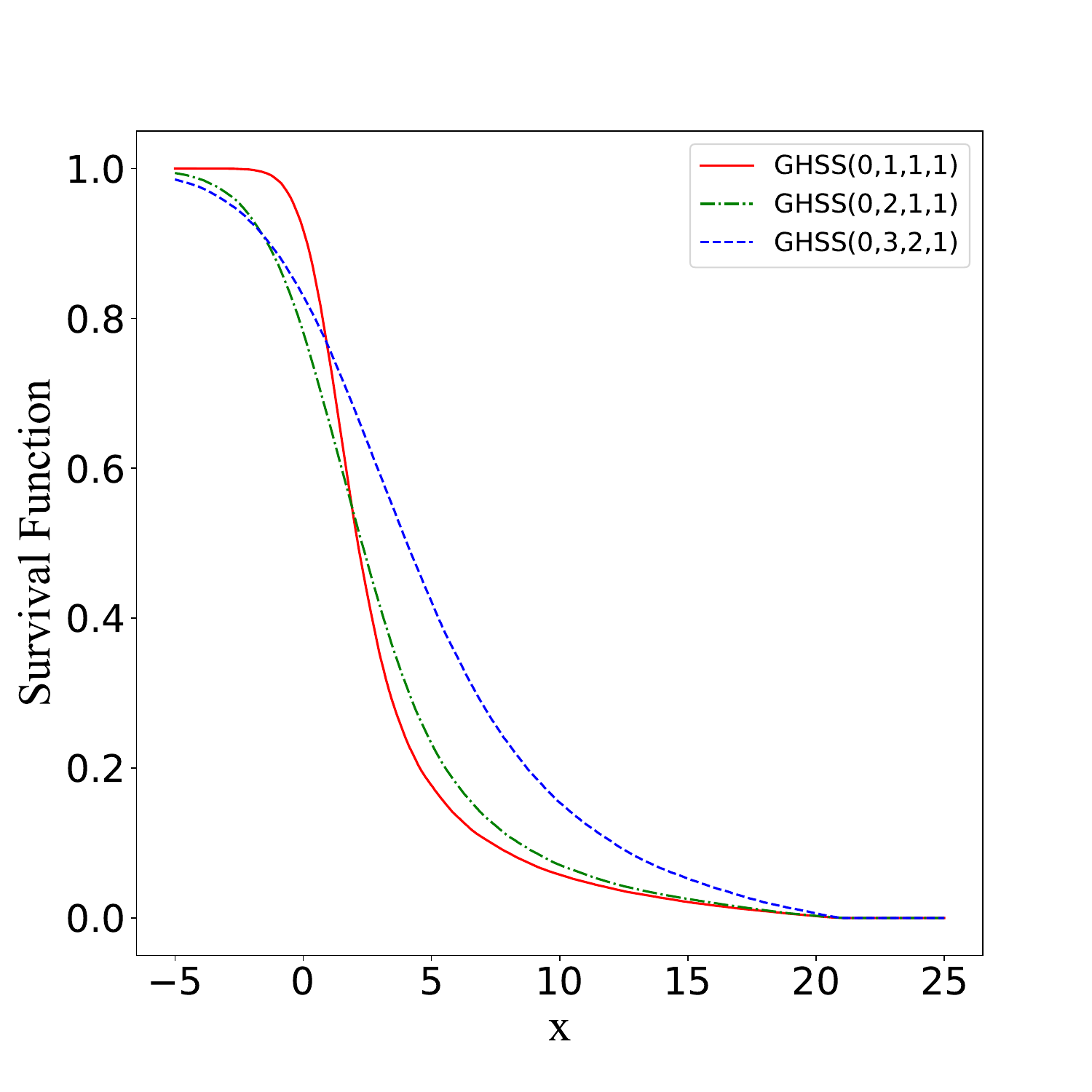}}}
  \caption{Survival functions of univariate GHSS distributions.} \label{fig:survival}
\end{figure}
The following theorem establishes sufficient and necessary conditions for two random vectors following GLSE distributions with different $\boldsymbol{\mu}$, $\boldsymbol{\Sigma}$, and $\boldsymbol{\delta}$.
\begin{theorem} \label{st-dimn}
Assume that
\begin{equation} \label{Y1Y2-assumption}
\begin{split}
&\boldsymbol{Y}_1 \sim {\rm GLSE}_n\left(\boldsymbol{\mu}_1, \boldsymbol{\Sigma}_1, \boldsymbol{\delta}_1, \psi, \alpha, \beta, H\right),   \\
&\boldsymbol{Y}_2 \sim {\rm GLSE}_n\left(\boldsymbol{\mu}_2, \boldsymbol{\Sigma}_2, \boldsymbol{\delta}_2, \psi, \alpha, \beta, H\right).
\end{split}
\end{equation}
\begin{enumerate}
\item If $\boldsymbol{\mu}_2 + \beta(\mathbf{z}) \boldsymbol{\delta}_2  \geq \boldsymbol{\mu}_1 + \beta(\mathbf{z}) \boldsymbol{\delta}_1$ for all $\mathbf{z}$ and $\boldsymbol{\Sigma}_1 = \boldsymbol{\Sigma}_2$, then $\boldsymbol{Y}_1 \leq_{st} \boldsymbol{Y}_2$.
\item If $\boldsymbol{Y}_1 \leq_{st} \boldsymbol{Y}_2$ and the corresponding density density generator $g_1$ for $\boldsymbol{a}^T\boldsymbol{Y}_1$ satisfies Assumption \ref{ass1} for all $\boldsymbol{a} \in \mathbb{R}^n$, then $\boldsymbol{\mu}_1+E\left(\beta(\mathbf{Z})\right)\boldsymbol{\delta}_1 \leq \boldsymbol{\mu}_2+E\left(\beta(\mathbf{Z})\right)\boldsymbol{\delta}_2$ and $\boldsymbol{\Sigma}_1 = \boldsymbol{\Sigma}_2$.
\end{enumerate}
\end{theorem}
\begin{proof}
1. The proof is routine and thus omitted.
\par 2. It follows from $\boldsymbol{Y}_1 \leq_{st} \boldsymbol{Y}_2$ that  $\boldsymbol{Y}_1 \leq_{plst} \boldsymbol{Y}_2$, which means $\boldsymbol{Y}_{1,i} \leq_{st} \boldsymbol{Y}_{2,i}$ and $\boldsymbol{Y}_{1,i} + \boldsymbol{Y}_{1,j} \leq_{st} \boldsymbol{Y}_{2,i} + \boldsymbol{Y}_{2,j}$ for all $1 \leq i,j \leq p$, where $\boldsymbol{Y}_{1,i}$($\boldsymbol{Y}_{2,i}$) stands for the $i$-th component of $\boldsymbol{Y}_{1}$($\boldsymbol{Y}_{2}$). Note that $\boldsymbol{Y}_1,\boldsymbol{Y}_2$ following (\ref{Y1Y2-assumption}) leads to
\begin{equation} \label{component-sim}
\boldsymbol{Y}_{1,i}  \sim  {\rm GLSE}_1\left({\mu}_{1,i}, {\sigma}_{1,ii}, {\delta}_{1,i}, \psi, \alpha, \beta, H\right) ,
\end{equation}
\begin{equation} \label{component2-sim}
\boldsymbol{Y}_{2,i}  \sim  {\rm GLSE}_1\left({\mu}_{2,i}, {\sigma}_{2,ii}, {\delta}_{2,i}, \psi, \alpha, \beta, H\right) ,
\end{equation}
\begin{equation} \nonumber
\boldsymbol{Y}_{1,i} +\boldsymbol{Y}_{1,j}  \sim {\rm GLSE}_1\left({\mu}_{1,i} + {\mu}_{1,j}, 2{\sigma}_{1,ij} + {\sigma}_{1,ii} + {\sigma}_{1,jj}, {\delta}_{1,i} + {\delta}_{1,j}, \psi, \alpha, \beta, H\right) .
\end{equation}
Applying Lemma \ref{st-dim1}, then the desired result is obtained.
\end{proof}
\begin{remark}
We know $\boldsymbol{Y}_1 \leq_{plst} \boldsymbol{Y}_2$ can be derived from $\boldsymbol{Y}_1 \leq_{st} \boldsymbol{Y}_2 $. Consequently, if one change ``$\leq_{st}$'' to ``$\leq_{plst}$'' in the second statement of Theorem \ref{st-dimn}, the result is still valid.
\end{remark}
The following results generalizes Theorem 3.2 in \cite{yin2019stochastic} and Proposition 5 in \cite{jamali2020integral} to the GLSE distribution case.
\begin{theorem} \label{cx}
Let $\boldsymbol{Y}_1,\boldsymbol{Y}_2$ follow (\ref{Y1Y2-assumption}). We have the following conclusions:
\begin{enumerate}
\item  If $\boldsymbol{\mu}_1 = \boldsymbol{\mu}_2$, $\boldsymbol{\delta}_1 = \boldsymbol{\delta}_2$ and $\boldsymbol{\Sigma}_2 - \boldsymbol{\Sigma}_1$ is positive semi-definite, then $\boldsymbol{Y}_1 \leq_{cx} (\leq_{lcx}, \leq_{ilcx}) \boldsymbol{Y}_2$.
\item If $\boldsymbol{\mu}_1 = \boldsymbol{\mu}_2$, then $\boldsymbol{Y}_1 \leq_{cx} (\leq_{lcx}, \leq_{ilcx}) \boldsymbol{Y}_2$ if and only if $\boldsymbol{\delta}_1 = \boldsymbol{\delta}_2$ and $\boldsymbol{\Sigma}_2 - \boldsymbol{\Sigma}_1$ is positive semi-definite.
\item If $\boldsymbol{\delta}_1 = \boldsymbol{\delta}_2$, then $\boldsymbol{Y}_1 \leq_{cx} (\leq_{lcx}, \leq_{ilcx}) \boldsymbol{Y}_2$ if and only if $\boldsymbol{\mu}_1 = \boldsymbol{\mu}_2$ and $\boldsymbol{\Sigma}_2 - \boldsymbol{\Sigma}_1$ is positive semi-definite.
\end{enumerate}
\end{theorem}
\begin{proof}
1.  As $\boldsymbol{\Sigma}_2 - \boldsymbol{\Sigma}_1$ is positive semi-definite, there
exists a matrix $\mathbf{A}_{n \times n}$ such that $\boldsymbol{\Sigma}_2 - \boldsymbol{\Sigma}_1 = \mathbf{AA}^T$. Suppose
$\mathbf{A}=\left( \boldsymbol{a}_1, \boldsymbol{a}_2, \dots, \boldsymbol{a}_{n}\right)$, where $\boldsymbol{a}_i$ is an $n$-dimensional column vector, for $i=1, 2, \dots, np$. One can notice that the Hessian matrix $\boldsymbol{H}_f\left( \mathbf{X}\right)$ for twice differentiable convex function $f$ is positive semi-definite, which shows that
\begin{equation} \nonumber
{\rm tr}\left( \left( \boldsymbol{\Sigma}_2 - \boldsymbol{\Sigma}_1 \right)\boldsymbol{H}_f\left( \mathbf{X}\right)\right) = {\rm tr}\left( \mathbf{A}^T \boldsymbol{H}_f\left( \mathbf{X}\right) \mathbf{A}\right) =\sum_{i=1}^{n} \boldsymbol{a}^T_i \boldsymbol{H}_f\left( \mathbf{X}\right) \boldsymbol{a}_i \ge 0.
\end{equation}
Then $E\left[ f\left(\boldsymbol{Y}_1\right) \right] - E\left[ f\left(\boldsymbol{Y}_2\right) \right] \geq 0$ for all convex function $f$ by applying Lemma \ref{id-lse}, i.e. $\boldsymbol{Y}_1 \leq_{cx} \boldsymbol{Y}_2$. $\boldsymbol{Y}_1 \leq_{lcx} \boldsymbol{Y}_2$ and $\boldsymbol{Y}_1 \leq_{ilcx} \boldsymbol{Y}_2$ can be easily derived from $\boldsymbol{Y}_1 \leq_{cx} \boldsymbol{Y}_2$ by chain of implications (\ref{im-chain}).
\par 2. \& 3. One the one hand, it can be derived from $\boldsymbol{Y}_1 \leq_{cx} \boldsymbol{Y}_2$ that $E\boldsymbol{Y}_1 = E\boldsymbol{Y}_2$; therefore, if we know $\boldsymbol{\mu}_1 = \boldsymbol{\mu}_2$, then $\boldsymbol{\delta}_1 = \boldsymbol{\delta}_2$ can be obtained and vice versa. On the other hand, $\boldsymbol{Y}_1 \leq_{cx} \boldsymbol{Y}_2$ implies $\boldsymbol{Y}_1 \leq_{lcx} \boldsymbol{Y}_2$. We claim that $\boldsymbol{\Sigma}_2 - \boldsymbol{\Sigma}_1$ is positive semi-definite. Otherwise, there exists $\boldsymbol{a} \in \mathbb{R}^n$ such that $\boldsymbol{a}^T \left(\boldsymbol{\Sigma}_2 - \boldsymbol{\Sigma}_1\right) \boldsymbol{a}<0$. Let $f(\mathbf{x}) = \left(\boldsymbol{a}^T\mathbf{x}\right)^2$, which is convex. According to Definition \ref{def-order}, we have $E\left( \boldsymbol{a}^T\boldsymbol{Y}_1\boldsymbol{Y}_1^T\boldsymbol{a} \right) \leq E\left( \boldsymbol{a}^T\boldsymbol{Y}_1\boldsymbol{Y}_1^T\boldsymbol{a} \right)$. It can be derived by considering (\ref{lse-cov}) that $\boldsymbol{a}^T \left(\boldsymbol{\Sigma}_2 - \boldsymbol{\Sigma}_1\right) \boldsymbol{a} \geq 0$, which leads to a contradiction. Chain of implications (\ref{im-chain}) shows that $\boldsymbol{Y}_1 \leq_{lcx} \boldsymbol{Y}_2$ if and only if $\boldsymbol{Y}_1 \leq_{ilcx} \boldsymbol{Y}_2$, the result is still valid if one changes $\boldsymbol{Y}_1 \leq_{lcx} \boldsymbol{Y}_2$ to $\boldsymbol{Y}_1 \leq_{ilcx} \boldsymbol{Y}_2$.
\end{proof}
\par The increasing convex order, also known as stop-loss order, is widely used in the area of actuarial science. The following theorem provides necessary and sufficient conditions for the increasing convex ordering of two univariate GLSE distributed random variables. Some related conditions for elliptical distributions can be found in \cite{pan2016stochastic}.
\begin{lemma} \label{icx-dim1}
Let $\boldsymbol{Y}_1,\boldsymbol{Y}_2$ follow (\ref{GLSEdim1}). We have the following conclusions:
\begin{enumerate}
\item If $\mu_2 - \mu_1 + \beta(\mathbf{z}) \left({\delta}_2 - {\delta}_1 \right) \geq 0$ for all $\mathbf{z}$ and ${\sigma}_1 \leq {\sigma}_2$, then ${Y}_1 \leq_{icx} {Y}_2$.
\item If ${Y}_1 \leq_{icx} {Y}_2$ and the corresponding density generator $g_1$ satisfies Assumption \ref{ass2}, then $\mu_2 - \mu_1 + E\beta(\boldsymbol{Z}) \left({\delta}_2 - {\delta}_1 \right) \geq 0$ and ${\sigma}_1 \leq {\sigma}_2$.
\end{enumerate}
\end{lemma}
\begin{proof}
1. The implication follows from Lemma \ref{id-lse}.
\par 2. Note that ${Y}_1 \leq_{icx} {Y}_2$ implies $E{Y}_1 \leq E{Y}_2$, and thus $\mu_2 - \mu_1 + E\beta(\boldsymbol{Z}) \left({\delta}_2 - {\delta}_1 \right)\geq 0$. We now show that ${\sigma}_1 \leq {\sigma}_2$. If ${\sigma}_1 > {\sigma}_2$, then $\overline{F}_2(t) < \overline{F}_1(t)$ for sufficiently large positive $t$, which can be proved according to the proof of Lemma \ref{st-dim1}. Then, for sufficiently large positive $t$, we have
\begin{equation} \nonumber
E({Y}_1 - t)_+ = \int_t^{+\infty} \overline{F}_1(x) dx > \int_t^{+\infty} \overline{F}_2(x) dx= E({Y}_2 - t)_+,
\end{equation}
which results in a contradiction to ${Y}_1 \leq_{icx} {Y}_2$.
\end{proof}
\begin{figure}
  \centering
  \subfloat[]{%
  \resizebox*{7cm}{!}{\includegraphics{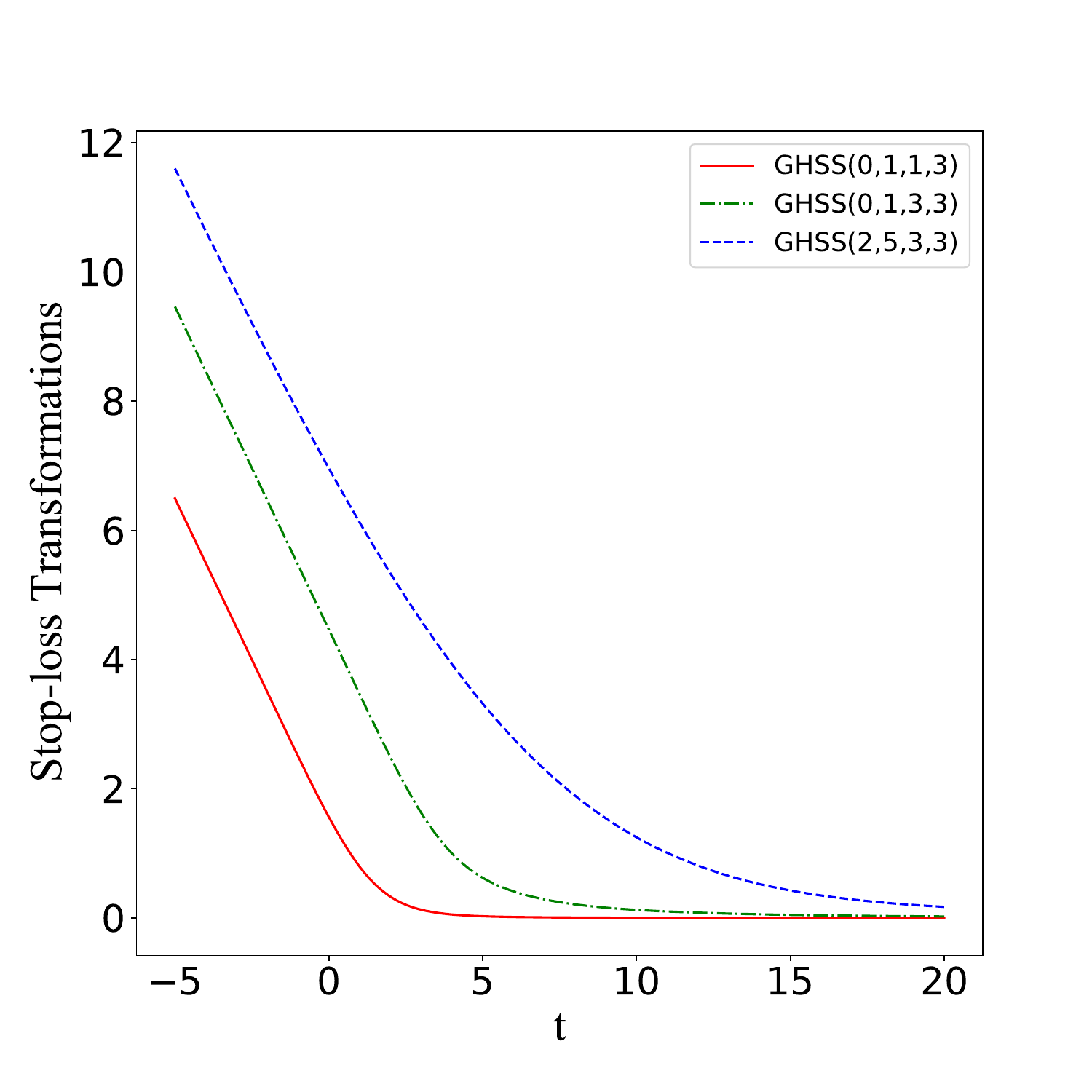}}}\hspace{5pt}
  \subfloat[]{%
  \resizebox*{7cm}{!}{\includegraphics{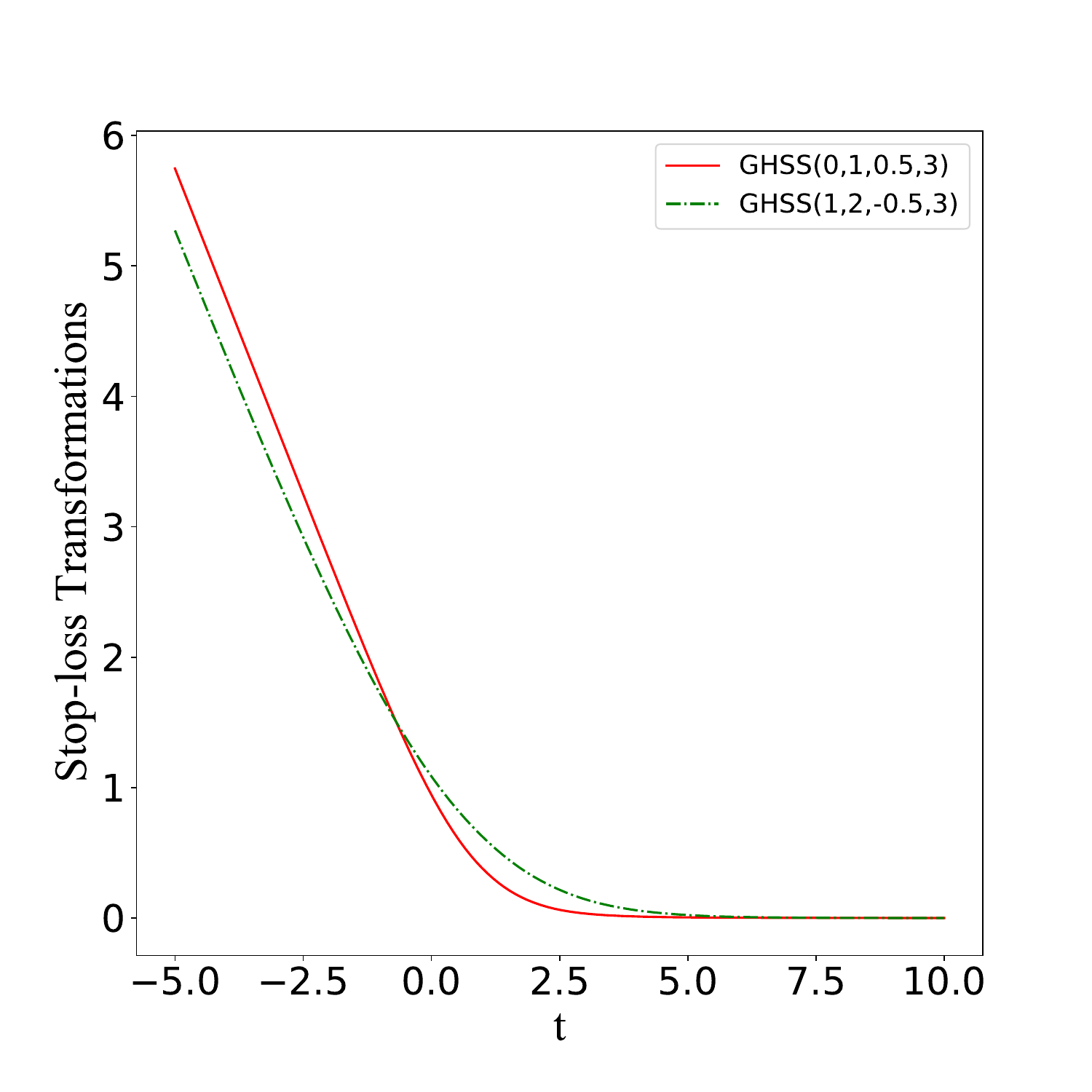}}}
  \caption{Stoploss transformations of univariate GHSS distributions} \label{fig:stoploss}
  \end{figure}

\begin{example}
As an illustration for Lemma \ref{icx-dim1}, the stop-loss transformations for the univariate GHSS distributions are plotted in Figure \ref{fig:stoploss}. It is easy to check that the conditions in Lemma \ref{icx-dim1}(1) are satisfied, and consistent with the increasing convex ordering among these three distributions displayed by the stop-loss transformations in Figure \ref{fig:stoploss}(a). Figure \ref{fig:stoploss}(b) provides a counterexample under the setting $(\mu_1,\sigma_1,\delta_1)=(0,1,0.5)$ and $(\mu_2,\sigma_2,\delta_2)=(1,2,-0.5)$. It is clear that size relation between ${\mu}_2 + \beta(\mathbf{z}) {\delta}_2 $ and ${\mu}_1 + \beta(\mathbf{z}) {\delta}_1$ varies with respect to $\mathbf{z}$, meaning that the conditions in Lemma \ref{icx-dim1}(1) are not fully satisfied. This agrees with the plots in Figure \ref{fig:stoploss}(b) that the stop-loss transformations of these distributions cross with each other.
\end{example}
Some necessary and sufficient conditions for some stochastic orders of multinormal mean-variance mixture were studied in \cite{jamali2020integral}; however, the sufficient conditions for the multivariate increasing convex order were not given. The following theorem fills this gap, and generalizes Theorem 7 in \cite{muller2001stochastic} and Theorem 3.3 in \cite{yin2019stochastic} to the case of GLSE distributions.

\begin{theorem} \label{icx}
Let $\boldsymbol{Y}_1,\boldsymbol{Y}_2$ follows (\ref{Y1Y2-assumption}). The following results hold:
\begin{enumerate}
\item If $\boldsymbol{\mu}_2 + \beta(\mathbf{z}) \boldsymbol{\delta}_2  \geq \boldsymbol{\mu}_1 + \beta(\mathbf{z}) \boldsymbol{\delta}_1$ for all $\mathbf{z}$ and $\boldsymbol{\Sigma}_2 - \boldsymbol{\Sigma}_1$ is positive semi-definite, then $\boldsymbol{Y}_1 \leq_{icx} (\leq_{iplcx}) \boldsymbol{Y}_2$.
\item If $\boldsymbol{Y}_1 \leq_{icx} (\leq_{iplcx}) \boldsymbol{Y}_2$ and the corresponding density generator $g_1$ of $\boldsymbol{a}^T\boldsymbol{Y}_1$ satisfies Assumption \ref{ass2} for all $\boldsymbol{a} \in \mathbb{R}_+^n$, then $\boldsymbol{\mu}_2 + E\beta(\boldsymbol{Z}) \boldsymbol{\delta}_2  \geq \boldsymbol{\mu}_1 + E\beta(\boldsymbol{Z}) \boldsymbol{\delta}_1$ and $\boldsymbol{\Sigma}_2 - \boldsymbol{\Sigma}_1$ is copositive.
\end{enumerate}
\end{theorem}
\begin{proof} 
1. The conditions $\boldsymbol{\mu}_2 + \beta(\mathbf{z}) \boldsymbol{\delta}_2  \geq \boldsymbol{\mu}_1 + \beta(\mathbf{z}) \boldsymbol{\delta}_1$ for all $\mathbf{z}$ and $\boldsymbol{\Sigma}_2 - \boldsymbol{\Sigma}_1$ is positive semi-definite imply $\boldsymbol{Y}_1 \leq_{icx} \boldsymbol{Y}_2$ by Lemma \ref{id-lse}. Then $\boldsymbol{Y}_1 \leq_{iplcx} \boldsymbol{Y}_2$ can be easily derived by chain of implications (\ref{im-chain}).

\par 2. Since $\boldsymbol{Y}_1 \leq_{icx} \boldsymbol{Y}_2$ ($\boldsymbol{Y}_1 \leq_{iplcx} \boldsymbol{Y}_2$) implies $E\boldsymbol{Y}_1 \leq E\boldsymbol{Y}_2$, and thus $\boldsymbol{\mu}_2 + E\beta(\boldsymbol{Z}) \boldsymbol{\delta}_2  \geq \boldsymbol{\mu}_1 + E\beta(\boldsymbol{Z}) \boldsymbol{\delta}_1$.
Moreover, $\boldsymbol{Y}_1 \leq_{icx} \boldsymbol{Y}_2$ implies $\boldsymbol{Y}_1 \leq_{iplcx} \boldsymbol{Y}_2$, which means $\boldsymbol{a}^T\boldsymbol{Y}_1 \leq_{icx} \boldsymbol{a}^T\boldsymbol{Y}_2$ is valid for all $\boldsymbol{a} \in \mathbb{R}_+^n$. From Lemma \ref{lse-aff}, we have $\boldsymbol{a}^T\boldsymbol{Y}_i \sim {\rm GLSE}_1\left( \boldsymbol{a}^T \boldsymbol{\mu}_i, \boldsymbol{a}^T\boldsymbol{\Sigma}_i\boldsymbol{a},\boldsymbol{a}^T\boldsymbol{\delta}_i,\psi,\alpha,\beta,H\right)$, $i=1,2$. Both $\boldsymbol{a}^T\boldsymbol{Y}_1$ and $\boldsymbol{a}^T\boldsymbol{Y}_2$ are quantities, and the density generator they share satisfies Assumption \ref{ass2}, thus, $\boldsymbol{a}^T\boldsymbol{\Sigma}_2\boldsymbol{a} - \boldsymbol{a}^T\boldsymbol{\Sigma}_1\boldsymbol{a} \geq 0$ can be derived from Lemma \ref{icx-dim1}.
\end{proof}
It is worth noting that $\boldsymbol{Y}_1 \leq_{icx} \boldsymbol{Y}_2$ does not necessarily imply $\boldsymbol{Y}_1 \leq_{ilcx} \boldsymbol{Y}_2$ for all distributions $\boldsymbol{Y}_1$ and $\boldsymbol{Y}_2$. A counterexample could be found in \cite{pan2016stochastic}. But $\boldsymbol{Y}_1 \leq_{icx} \boldsymbol{Y}_2$ does imply $\boldsymbol{Y}_1 \leq_{iplcx} \boldsymbol{Y}_2$ because $g_a(\boldsymbol{x}) = g(\boldsymbol{a}^T \boldsymbol{x})$ is increasing convex for any increasing convex function $g$ and $\boldsymbol{a} \in \mathbb{R}_+^n$.
The following result generalizes Theorem 12 in \cite{muller2001stochastic}, Theorem 3.6 in \cite{yin2019stochastic} and Proposition 6 in \cite{Amiri2020Linear}.
\begin{theorem}
Let $\boldsymbol{Y}_1,\boldsymbol{Y}_2$ follow (\ref{Y1Y2-assumption}). The following statements are true:
\begin{enumerate}
\item If $\boldsymbol{\mu}_1 = \boldsymbol{\mu}_2$, $\boldsymbol{\delta}_1 = \boldsymbol{\delta}_2$ and $\boldsymbol{\Sigma}_2 \geq \boldsymbol{\Sigma}_1$, then $\boldsymbol{Y}_1 \leq_{dcx} \boldsymbol{Y}_2$.
\item If $\boldsymbol{\mu}_1 = \boldsymbol{\mu}_2$, then $\boldsymbol{Y}_1 \leq_{dcx} \boldsymbol{Y}_2$ if and only if $\boldsymbol{\delta}_1 = \boldsymbol{\delta}_2$ and $\boldsymbol{\Sigma}_2 \geq \boldsymbol{\Sigma}_1$.
\item If $\boldsymbol{\delta}_1 = \boldsymbol{\delta}_2$, then $\boldsymbol{Y}_1 \leq_{dcx} \boldsymbol{Y}_2$ if and only if $\boldsymbol{\mu}_1 = \boldsymbol{\mu}_2$ and $\boldsymbol{\Sigma}_2 \geq \boldsymbol{\Sigma}_1$.
\end{enumerate}
\end{theorem}
\begin{proof}
1. The proof is routine and thus omitted.
\par 2. \& 3. Note that the functions $f_1(\mathbf{x}) = x_i$ and $f_2(\mathbf{x}) = - x_i $ are directionally convex for all $1 \leq i \leq n$. Therefore, $E\boldsymbol{Y}_1 = E\boldsymbol{Y}_2$. Then the equivalence between $\boldsymbol{\delta}_1$ and $\boldsymbol{\delta}_2$ (alternatively, $\boldsymbol{\mu}_1$ and $\boldsymbol{\mu}_2$) can be established by using the same method in the proof of Theorem \ref{cx}.
\par Let $f_3(\mathbf{x}) = x_i x_j$, which is directionally convex for all $1 \leq i,j \leq n$. It can be derived that $Cov\left(\boldsymbol{Y}_1\right) \leq Cov\left(\boldsymbol{Y}_2\right)$, then we claim $\boldsymbol{\Sigma}_2 \geq \boldsymbol{\Sigma}_1$ on the ground that $\boldsymbol{\delta}_1 = \boldsymbol{\delta}_2$.
\end{proof}

The following theorem considers the componentwise convex order. As some special cases, the multinormal case can be found in \cite{arlotto2009hessian} while the multivariate elliptical case can be found in \cite{yin2019stochastic}.
\begin{theorem}
Let $\boldsymbol{Y}_1,\boldsymbol{Y}_2$ follow (\ref{Y1Y2-assumption}). The following statements hold:
\begin{enumerate}
\item If $\boldsymbol{\mu}_1 = \boldsymbol{\mu}_2$, $\boldsymbol{\delta}_1 = \boldsymbol{\delta}_2$, $\sigma_{1,ii} \leq \sigma_{2,ii}$ for $1 \leq i \leq n$ and $\sigma_{1,ij} = \sigma_{2,ij}$ for $1 \leq i<j \leq n$, then $\boldsymbol{Y}_1 \leq_{ccx} \boldsymbol{Y}_2$.
\item If $\boldsymbol{\mu}_1 = \boldsymbol{\mu}_2$, then $\boldsymbol{Y}_1 \leq_{ccx} \boldsymbol{Y}_2$  if and only if $\boldsymbol{\delta}_1 = \boldsymbol{\delta}_2$, $\sigma_{1,ii} \leq \sigma_{2,ii}$ for $1 \leq i \leq n$ and $\sigma_{1,ij} = \sigma_{2,ij}$ for $1 \leq i<j \leq n$.
\item If $\boldsymbol{\delta}_1 = \boldsymbol{\delta}_2$, then $\boldsymbol{Y}_1 \leq_{ccx} \boldsymbol{Y}_2$  if and only if $\boldsymbol{\mu}_1 = \boldsymbol{\mu}_2$, $\sigma_{1,ii} \leq \sigma_{2,ii}$ for $1 \leq i \leq n$ and $\sigma_{1,ij} = \sigma_{2,ij}$ for $1 \leq i<j \leq n$.
\end{enumerate}
\end{theorem}
\begin{proof}
1. The proof is routine and thus omitted.
\par 2. \& 3. Note that the functions $f_1(\mathbf{x}) = x_i$ and $f_2(\mathbf{x}) = - x_i $ are componentwise convex for all $1 \leq i \leq n$, by applying which one has $E\boldsymbol{Y}_1 = E\boldsymbol{Y}_2$. Then the equivalence between $\boldsymbol{\delta}_1$ and $\boldsymbol{\delta}_2$ (alternatively, $\boldsymbol{\mu}_1$ and $\boldsymbol{\mu}_2$) can be established by using the same method in the proof of Theorem \ref{cx}.
\par Let $f_3(\mathbf{x}) = x_i x_j$, $f_4(\mathbf{x}) = - x_i x_j$ and $f_5(\mathbf{x}) = x_i^2$. Clearly, all of them are componentwise convex for all $1 \leq i < j \leq n$. Thus, we get $\sigma_{1,ii} \leq \sigma_{2,ii}$ for $1 \leq i \leq n$ and $\sigma_{1,ij} = \sigma_{2,ij}$ for $1 \leq i<j \leq n$ by considering (\ref{lse-cov}).
\end{proof}
Supermodular orders are important for a wide range of scientific and industrial processes. Several practical applications for supermodular orders, like applications in genetic selection, are presented in \cite{Bauerle1997}. The following result generalizes Theorem 11 in \cite{muller2001stochastic} from the multivariate normal case to the GLSE setting.
\begin{theorem} \label{sm}
Let $\boldsymbol{Y}_1,\boldsymbol{Y}_2$ follow (\ref{Y1Y2-assumption}).
$\boldsymbol{Y}_1 \leq_{sm} \boldsymbol{Y}_2$ if and only if $\boldsymbol{Y}_1$ and $\boldsymbol{Y}_2$ have the same marginals and $\sigma_{1,ij} \leq \sigma_{2,ij}$ for all $1 \leq i < j \leq n$.
\end{theorem}
\begin{proof}
Suppose $\boldsymbol{Y}_1 \leq_{sm} \boldsymbol{Y}_2$. It can hold only if the random vectors have the same marginals, which means $\boldsymbol{\mu}_1 = \boldsymbol{\mu}_2$, $\boldsymbol{\delta}_1 = \boldsymbol{\delta}_2$ and $\sigma_{1,ii} = \sigma_{2,ii}$ for any $1 \leq i \leq n$. Since the function $f(\mathbf{x}) = x_i x_j$ is supermodular for all $1 \leq i \neq j \leq n$, we see $\boldsymbol{Y}_1 \leq_{sm} \boldsymbol{Y}_2$ implies $\sigma_{1,ij} \leq \sigma_{2,ij}$ for all $1 \leq i \neq j \leq n$. Lemma \ref{id-lse} yields the converse, and hence the result follows.
\end{proof}
From the perspective of correlation, Theorem \ref{sm} can be presented as follows.
\begin{corollary} \label{sm2}
Let $\boldsymbol{Y}_1,\boldsymbol{Y}_2$ follow (\ref{Y1Y2-assumption}), where $\boldsymbol{\Sigma}_1$ and $\boldsymbol{\Sigma}_2$ are correlation matrices. Then $\boldsymbol{Y}_1 \leq_{sm} \boldsymbol{Y}_2$ if and only if $\boldsymbol{\mu}_1 = \boldsymbol{\mu}_2$, $\boldsymbol{\delta}_1 = \boldsymbol{\delta}_2$ and $\sigma_{1,ij} \leq \sigma_{2,ij}$ for all $1 \leq i < j \leq n$.
\end{corollary}
The upper orthant order, given in Definition \ref{def-order}, can also be defined through a comparison of upper orthants, which means that $\boldsymbol{Y}_1 \leq_{uo} \boldsymbol{Y}_2$ if and only if $P(\boldsymbol{Y}_1 > \boldsymbol{t}) \leq P(\boldsymbol{Y}_2 > \boldsymbol{t})$ holds for all $\boldsymbol{t}$. These two definitions can be shown to be equivalent. The following lemma, which is presented in \cite{muller2000some}, provides the fact that there is no difference between the upper orthant order and the supermodular order in the bivariate case.
\begin{lemma} \label{sm=uo}
\citep{muller2000some} Let $\boldsymbol{X}$, $\boldsymbol{Y}$ be two bivariate random vectors and have the same marginals, then $\boldsymbol{X} \leq_{sm} \boldsymbol{Y}$ is equivalent to $\boldsymbol{X} \leq_{uo} \boldsymbol{Y}$.
\end{lemma}
The inequality version of Lemma \ref{sm=uo} can be found in \cite{tong} and \cite{Ruschendorf}. The following theorem provides conditions for comparing GLSE distributed vectors under the upper orthant order. The upper orthant order can be equivalently defined by requiring $\overline{F}_X(\mathbf{t}) \leq \overline{F}_Y(\mathbf{t})$ for all $\mathbf{t} \in \mathbb{R}^n$.
\begin{theorem} \label{uo}
Let $\boldsymbol{Y}_1,\boldsymbol{Y}_2$ follow (\ref{Y1Y2-assumption}). The following statements are true:
\begin{enumerate}
\item If $\boldsymbol{\mu}_2 + \beta(\mathbf{z}) \boldsymbol{\delta}_2  \geq \boldsymbol{\mu}_1 + \beta(\mathbf{z}) \boldsymbol{\delta}_1$ for all $\mathbf{z}$, $\sigma_{1,ii} = \sigma_{2,ii}$ for all $1 \leq i \leq n$ and $\sigma_{1,ij} \leq \sigma_{2,ij}$ for all $1 \leq i \neq j \leq n$, then $\boldsymbol{Y}_1 \leq_{uo} \boldsymbol{Y}_2$.
\item If $\boldsymbol{Y}_1 \leq_{uo} \boldsymbol{Y}_2$ and $g$ satisfies Assumption \ref{ass1}, then $\boldsymbol{\mu}_1+E\left(\beta(\mathbf{Z})\right)\boldsymbol{\delta}_1 \leq \boldsymbol{\mu}_2+E\left(\beta(\mathbf{Z})\right)\boldsymbol{\delta}_2$ and $\sigma_{1,ii} = \sigma_{2,ii}$ for all $1 \leq i \leq n$.
\item If $\boldsymbol{Y}_1$ and $\boldsymbol{Y}_2$ have the same marginals and $\boldsymbol{Y}_1 \leq_{uo} \boldsymbol{Y}_2$, then $\sigma_{1,ij} \leq \sigma_{2,ij}$ for all $1 \leq i < j \leq n$.
\end{enumerate}
\end{theorem}
\begin{proof}
1. For any $\boldsymbol{\Delta}$-monotone function $f$, $\nabla f(\mathbf{x}) \geq 0$ and the  off-diagonal elements in $\mathbf{H}_f(\mathbf{x})$ are greater than $0$. Then the results can be derived by using Lemma \ref{id-lse}.
\par 2.  $\boldsymbol{Y}_1 \leq_{uo} \boldsymbol{Y}_2$ implies their components $\boldsymbol{Y}_{1,i} \leq_{st} \boldsymbol{Y}_{2,i}$ for all $1 \leq i \leq n$. Note that $\boldsymbol{Y}_1,\boldsymbol{Y}_2$ following (\ref{Y1Y2-assumption}) leads to (\ref{component-sim}) and (\ref{component2-sim}). Then the desired results can be proved by applying Lemma \ref{st-dim1}.
\par 3. $\boldsymbol{Y}_1 \leq_{uo} \boldsymbol{Y}_2$ implies that $\left( \boldsymbol{Y}_{1,i},\boldsymbol{Y}_{1,j}\right)^T \leq_{uo} \left( \boldsymbol{Y}_{1,i},\boldsymbol{Y}_{1,j}\right)^T$, where $1 \leq i < j \leq n$ and it is quite obvious that $\left( \boldsymbol{Y}_{1,i},\boldsymbol{Y}_{1,j}\right)^T$ and $\left( \boldsymbol{Y}_{2,i},\boldsymbol{Y}_{2,j}\right)^T$ have the same marginals. It can be derived from Lemma \ref{sm=uo} that $\left( \boldsymbol{Y}_{1,i},\boldsymbol{Y}_{1,j}\right)^T \leq_{sm} \left( \boldsymbol{Y}_{1,i},\boldsymbol{Y}_{1,j}\right)^T$. Then the required result follows from Theorem \ref{sm}.
\end{proof}
At the end of this section, we will consider the copositive and completely positive orders for random vectors following multivariate GLSE distribution. The multivariate normal and elliptical cases can be found in \cite{arlotto2009hessian} and \cite{yin2019stochastic}.
\begin{theorem} \label{cp}
Let $\boldsymbol{Y}_1,\boldsymbol{Y}_2$ follow (\ref{Y1Y2-assumption}). The following statements are true:
\begin{enumerate}
\item If $\boldsymbol{\mu}_1 = \boldsymbol{\mu}_2$, $\boldsymbol{\delta}_1 = \boldsymbol{\delta}_2$ and $\boldsymbol{\Sigma}_2 - \boldsymbol{\Sigma}_1$ is copositive, then $\boldsymbol{Y}_1 \leq_{cp} \boldsymbol{Y}_2$.
\item If $\boldsymbol{\mu}_1 = \boldsymbol{\mu}_2$, then $\boldsymbol{Y}_1 \leq_{cp} \boldsymbol{Y}_2$, if and only if $\boldsymbol{\delta}_1 = \boldsymbol{\delta}_2$ and $\boldsymbol{\Sigma}_2 - \boldsymbol{\Sigma}_1$ is copositive.
\item If $\boldsymbol{\delta}_1 = \boldsymbol{\delta}_2$, then $\boldsymbol{Y}_1 \leq_{cp} \boldsymbol{Y}_2$, if and only if $\boldsymbol{\mu}_1 = \boldsymbol{\mu}_2$ and $\boldsymbol{\Sigma}_2 - \boldsymbol{\Sigma}_1$ is copositive.
\end{enumerate}
\end{theorem}
\begin{proof}
1. For any function $f$ such that $\mathbf{H}_f(\mathbf{x}) \in \mathcal{C}_{cp}$, using Lemma \ref{id-lse}, together with (\ref{rela-cpcop}), it yields $Ef(\boldsymbol{Y}_1) \leq f(\boldsymbol{Y}_2)$, which means that $\boldsymbol{Y}_1 \leq_{cp} \boldsymbol{Y}_2$.
\par 2.\& 3.
Note that the Hessian matrices of functions $f_{1,i}(\mathbf{x}) = x_i$ and $f_{2,i}(\mathbf{x}) = - x_i $ are completely positive for all $1 \leq i \leq n$. Thus, $\boldsymbol{\mu}_2 + E\beta(\boldsymbol{Z}) \boldsymbol{\delta}_2  = \boldsymbol{\mu}_1 + E\beta(\boldsymbol{Z}) \boldsymbol{\delta}_1$ can be derived by setting $f=f_{1,i}$ and $f=f_{2,i}$ in Definition \ref{def-order}. If we know $\boldsymbol{\mu}_1 = \boldsymbol{\mu}_2$, then $\boldsymbol{\delta}_1 = \boldsymbol{\delta}_2$ can be obtained as well and vice versa. For any symmetric $n \times n$ matrix $\mathbf{A} \in \mathcal{C}_{cp}$, let
\begin{equation}\nonumber
f_{6}(\mathbf{x}) = \frac{1}{2}(\mathbf{x} - E\boldsymbol{Y}_1)^T\mathbf{A}(\mathbf{x} - E\boldsymbol{Y}_1).
\end{equation}
Notice the fact that the Hessian matrices of $f_{6}$ are $\mathbf{A}$ for all $\mathbf{x} \in \mathbb{R}^n$. To obtain the desired result, it suffices to show that
\begin{equation} \label{exp-f6}
E\left( (\boldsymbol{Y}_1 - E\boldsymbol{Y}_1)^T\mathbf{A}(\boldsymbol{Y}_1 - E\boldsymbol{Y}_1)\right) \leq E \left( (\boldsymbol{Y}_2 - E\boldsymbol{Y}_2)^T\mathbf{A}(\boldsymbol{Y}_2 - E\boldsymbol{Y}_2)\right),
\end{equation}
which can be obtained by setting $f=f_{6}$ in Definition \ref{def-order}.
The inequality (\ref{exp-f6}), together with (\ref{lse-cov}), allows us to get ${\rm tr}\left( (\boldsymbol{\Sigma}_2 - \boldsymbol{\Sigma}_1) \mathbf{A}\right) \geq 0$. Since $\mathbf{A} \in \mathcal{C}_{cp}$ is arbitrarily chosen, we conclude that $\boldsymbol{\Sigma}_2 - \boldsymbol{\Sigma}_1 \in \mathcal{C}_{cp}^*$, i.e. $\boldsymbol{\Sigma}_2 - \boldsymbol{\Sigma}_1$ is copositive.
\end{proof}
\begin{theorem} \label{cop}
Let $\boldsymbol{Y}_1,\boldsymbol{Y}_2$ follow (\ref{Y1Y2-assumption}). The following statements are true:
\begin{enumerate}
\item If $\boldsymbol{\mu}_1 = \boldsymbol{\mu}_2$, $\boldsymbol{\delta}_1 = \boldsymbol{\delta}_2$ and $\boldsymbol{\Sigma}_2 - \boldsymbol{\Sigma}_1$ is completely positive, then $\boldsymbol{Y}_1 \leq_{cop} \boldsymbol{Y}_2$.
\item If $\boldsymbol{\mu}_1 = \boldsymbol{\mu}_2$, then $\boldsymbol{Y}_1 \leq_{cop} \boldsymbol{Y}_2$, if and only if $\boldsymbol{\delta}_1 = \boldsymbol{\delta}_2$ and $\boldsymbol{\Sigma}_2 - \boldsymbol{\Sigma}_1$ is completely positive.
\item If $\boldsymbol{\delta}_1 = \boldsymbol{\delta}_2$, then $\boldsymbol{Y}_1 \leq_{cop} \boldsymbol{Y}_2$, if and only if $\boldsymbol{\mu}_1 = \boldsymbol{\mu}_2$ and $\boldsymbol{\Sigma}_2 - \boldsymbol{\Sigma}_1$ is completely positive.
\end{enumerate}
\end{theorem}
\begin{proof}
  1. For any function $f$ such that $\mathbf{H}_f(\mathbf{x}) \in \mathcal{C}_{cop}$, it follows from Lemma \ref{id-lse} and (\ref{rela-cpcop}) that $Ef(\boldsymbol{Y}_1) \leq f(\boldsymbol{Y}_2)$, which means that $\boldsymbol{Y}_1 \leq_{cp} \boldsymbol{Y}_2$.
  \par 2.\& 3.
  Note that the Hessian matrices of functions $f_{1,i}(\mathbf{x}) = x_i$ and $f_{2,i}(\mathbf{x}) = - x_i $ are completely positive for all $1 \leq i \leq n$. Thus, $\boldsymbol{\mu}_2 + E\beta(\boldsymbol{Z}) \boldsymbol{\delta}_2  = \boldsymbol{\mu}_1 + E\beta(\boldsymbol{Z}) \boldsymbol{\delta}_1$ can be derived by setting $f=f_{1,i}$ and $f=f_{2,i}$ in Definition \ref{def-order}. If we know $\boldsymbol{\mu}_1 = \boldsymbol{\mu}_2$, then $\boldsymbol{\delta}_1 = \boldsymbol{\delta}_2$ can be obtained as well and vice versa. For any symmetric $n \times n$ matrix $\mathbf{A} \in \mathcal{C}_{cop}$, let
  \begin{equation} \nonumber
  f_{7}(\mathbf{x}) = \frac{1}{2}(\mathbf{x} - E\boldsymbol{Y}_1)^T\mathbf{A}(\mathbf{x} - E\boldsymbol{Y}_1).
  \end{equation}
  Notice the fact that the Hessian matrices of $f_{7}$ are $\mathbf{A}$ for all $\mathbf{x} \in \mathbb{R}^n$. To obtain the desired result, it suffices to show that
  \begin{equation} \label{exp-f7}
  E\left( (\boldsymbol{Y}_1 - E\boldsymbol{Y}_1)^T\mathbf{A}(\boldsymbol{Y}_1 - E\boldsymbol{Y}_1)\right) \leq E \left( (\boldsymbol{Y}_2 - E\boldsymbol{Y}_2)^T\mathbf{A}(\boldsymbol{Y}_2 - E\boldsymbol{Y}_2)\right),
  \end{equation}
  which indeed can be obtained by setting $f=f_{7}$ in Definition \ref{def-order}.
  The inequality (\ref{exp-f7}), together with (\ref{lse-cov}), allows us to get ${\rm tr}\left( (\boldsymbol{\Sigma}_2 - \boldsymbol{\Sigma}_1) \mathbf{A}\right) \geq 0$. Since $\mathbf{A} \in \mathcal{C}_{cop}$ is arbitrarily chosen, we conclude that $\boldsymbol{\Sigma}_2 - \boldsymbol{\Sigma}_1 \in \mathcal{C}_{cop}^*$, i.e. $\boldsymbol{\Sigma}_2 - \boldsymbol{\Sigma}_1$ is completely positive.
\end{proof}
\section{Some Applications} \label{Applications}
In this part, we present some inequalities for certain functions of GLSE random variables, which can be proven by applying previous results. These inequalities are pretty valuable not only in extending the well-known Slepian's theorem to a general setting, but also in some actuarial practices for comparing the aggregate risks and the maximum claim amounts of two insurance portfolios.
\subsection{Extension of Slepian's Theorem}
The  Slepian's theorem is widely used in reliability theory, extreme value theory and pure probability. It was first introduced and proven by \cite{Slepian1962}, and used to compare tail behaviors of two normal distributions.  \cite{Gupta1972} generalized Slepian's theorem to the elliptical distributions. Extensions on Slepian's theorem for multivariate normal distributions with nonsingular covariance matrix and mean-variance mixtures of normal distributions can be found in \cite{topkis} and \cite{jamali2020integral}, respectively.  The class of GLSE distributions contains many important distributions as special cases; as a result, the following result is of high applicability and develops a generalization of Slepian's theorem for GLSE distributions, which is an  immediate consequence of Theorems \ref{sm} and \ref{uo}.
\begin{corollary}
Let $\boldsymbol{Y}_1,\boldsymbol{Y}_2$ follow (\ref{Y1Y2-assumption}). If $\boldsymbol{\mu}_1 \leq \boldsymbol{\mu}_2$, $\boldsymbol{\delta}_1 \leq \boldsymbol{\delta}_2$, $\sigma_{1,ii} = \sigma_{2,ii}$ for all $1 \leq i \leq n$ and $\sigma_{1,ij} \leq \sigma_{2,ij}$ for all $1 \leq i \neq j \leq n$, then
\begin{equation} \nonumber
P \left(\boldsymbol{Y}_1 > \boldsymbol{a}\right) \leq P \left(\boldsymbol{Y}_2 > \boldsymbol{a}\right)
\end{equation}
holds for all $\boldsymbol{a} \in \mathbb{R}^n$. Furthermore, the foregoing inequality is strict if $\boldsymbol{\mu}_1 < \boldsymbol{\mu}_2$, $\boldsymbol{\delta}_1 < \boldsymbol{\delta}_2$, and $\sigma_{1,ij} < \sigma_{2,ij}$ for all $1 \leq i \neq j \leq n$.
\end{corollary}
\subsection{Applications in Risk Models}
In this subsection, we provide some applications of the theoretical findings in some insurance scenarios of actuarial science. We shall consider three important quantities in individual and collective risk models including the aggregate claim amount, the maximum claim amount, and the Gini index. For discussions and stochastic orderings of three important quantities, we refer to \cite{zhang2015}, \cite{zhang2019}, \cite{Samanthi2016}, and \cite{Amiri2022Hessian}.
\par The aggregate claim amount in a particular time period is a quantity of fundamental importance for proper management of an insurance company in pricing insurance coverages. Given the claim amount of the $i$-th insurance contract $\boldsymbol{X}_{i}$ and the respective weights, $\alpha_i$, of the $i$-th insurance contract, for $i \in \lbrace 1, 2, \cdots, n\rbrace$, the individual risk model sets $\boldsymbol{S} = \sum_{i=1}^n \alpha_{i} \boldsymbol{X}_{i}$ as the aggregate risk. Consider two insurance portfolios with claims $\boldsymbol{Y}_1$ and $\boldsymbol{Y}_2$ assembled with same weights $\{\alpha_i, i=1,\ldots,n\}$. Let $\boldsymbol{S}_1 = \sum_{i=1}^n \alpha_{i} \boldsymbol{Y}_{1,i}$ and $\boldsymbol{S}_2 = \sum_{i=1}^n \alpha_{i} \boldsymbol{Y}_{2,i}$ be the aggregate risks of the two insurance portfolios. The following result is a direct consequence of Theorem \ref{st-dimn}, providing sufficient conditions for comparing the aggregated claims in individual risk model.
\begin{corollary} \label{nohatS}
Assume $\boldsymbol{Y}_1,\boldsymbol{Y}_2$ follow (\ref{Y1Y2-assumption}), then
\begin{enumerate}
  \item If $\boldsymbol{\mu}_2 + \beta(\mathbf{z}) \boldsymbol{\delta}_2  \geq \boldsymbol{\mu}_1 + \beta(\mathbf{z}) \boldsymbol{\delta}_1$ for all $\mathbf{z}$ and $\boldsymbol{\Sigma}_1 = \boldsymbol{\Sigma}_2$, then $\boldsymbol{S}_1 \leq_{st} \boldsymbol{S}_2$.
  \item If $\boldsymbol{\mu}_2 + \beta(\mathbf{z}) \boldsymbol{\delta}_2  \geq \boldsymbol{\mu}_1 + \beta(\mathbf{z}) \boldsymbol{\delta}_1$ for all $\mathbf{z}$ and $\boldsymbol{\Sigma}_2 - \boldsymbol{\Sigma}_1$ is positive semi-definite, then $\boldsymbol{S}_1 \leq_{icx} \boldsymbol{S}_2$.
\end{enumerate}
\end{corollary}
\begin{proof}
\begin{enumerate}
\item If $\boldsymbol{\mu}_2 + \beta(\mathbf{z}) \boldsymbol{\delta}_2  \geq \boldsymbol{\mu}_1 + \beta(\mathbf{z}) \boldsymbol{\delta}_1$ for all $\mathbf{z}$ and $\boldsymbol{\Sigma}_1 = \boldsymbol{\Sigma}_2$, then $\boldsymbol{Y}_1 \leq_{st} \boldsymbol{Y}_2$. Notice that $g(\mathbf{x}) = \sum_{i=1}^n \alpha_{i} X_{i}$ is increasing, so $f \circ g$ is increasing for all increasing $f$ as well. Then we have $Ef \circ g(\boldsymbol{Y}_1) \leq Ef \circ g(\boldsymbol{Y}_2)$, which can easily establish $\boldsymbol{S}_1 \leq_{st} \boldsymbol{S}_2$.
\item Notice that $g(\mathbf{x}) = \sum_{i=1}^n \alpha_{i} X_{i}$ is increasing and convex, so $f \circ g$ is increasing and convex for all increasing convex $f$ as well. The desired result can be established by applying the same method.
\end{enumerate}
\end{proof}


Besides the individual risk model, the collective risk model is a frequently-used tool to represent the aggregate risk of an insurance portfolio with random number of risks. The collective risk model considers $\hat{\boldsymbol{S}} = \sum_{i=1}^N \alpha_{i} \boldsymbol{X}_{i}$ as the aggregated risk of the insurance portfolio, where $\boldsymbol{N}$ is a counting random variable representing the number of policies in the portfolio which is independent of $\lbrace \boldsymbol{X}_{i}, i \in \lbrace 1, 2, \cdots, n\rbrace \rbrace$. Let $\hat{\boldsymbol{S}}_1 = \sum_{i=1}^{\boldsymbol{N}_1} \alpha_{i} \boldsymbol{Y}_{1,i}$ and $\hat{\boldsymbol{S}}_2 = \sum_{i=1}^{\boldsymbol{N}_2} \alpha_{i} \boldsymbol{Y}_{2,i}$ be the aggregate risks from two sets of insurance portfolios. The following corollary provides sufficient conditions for comparing the aggregated risk under the collective risk model.
\begin{corollary} \label{hatS}
Assume $\boldsymbol{Y}_1,\boldsymbol{Y}_2$ follow (\ref{Y1Y2-assumption}), random quantity $\boldsymbol{N}_1$ is independent of $\lbrace \boldsymbol{X}_{i}, i \in \lbrace 1, 2, \cdots, n\rbrace \rbrace$, random quantity $\boldsymbol{N}_2$ is independent of $\lbrace \boldsymbol{Y}_{i}, i \in \lbrace 1, 2, \cdots, n\rbrace \rbrace$, and $\boldsymbol{N}_1 \leq_{st} \boldsymbol{N}_2$, then
\begin{enumerate}
\item If $\boldsymbol{\mu}_2 + \beta(\mathbf{z}) \boldsymbol{\delta}_2  \geq \boldsymbol{\mu}_1 + \beta(\mathbf{z}) \boldsymbol{\delta}_1$ for all $\mathbf{z}$ and $\boldsymbol{\Sigma}_1 = \boldsymbol{\Sigma}_2$, then $\hat{\boldsymbol{S}}_1 \leq_{st} \hat{\boldsymbol{S}}_2$.
\item If $\boldsymbol{\mu}_2 + \beta(\mathbf{z}) \boldsymbol{\delta}_2  \geq \boldsymbol{\mu}_1 + \beta(\mathbf{z}) \boldsymbol{\delta}_1$ for all $\mathbf{z}$ and $\boldsymbol{\Sigma}_2 - \boldsymbol{\Sigma}_1$ is positive semi-definite, then $\hat{\boldsymbol{S}}_1 \leq_{icx} \hat{\boldsymbol{S}}_2$.
\end{enumerate}
\end{corollary}
\begin{proof} The proof can be easily established by combining Property 3.3.31 in \cite{denuit2006actuarial} and Corollary \ref{nohatS}.
\end{proof}
We next consider comparing the maximum claim amounts from the two insurance portfolios $\boldsymbol{Y}_1$ and $\boldsymbol{Y}_2$. Denote $\boldsymbol{M}_1 = \max \lbrace \boldsymbol{Y}_{1,1},\boldsymbol{Y}_{1,2}, \cdots, \boldsymbol{Y}_{1,n}\rbrace$, and $\boldsymbol{M}_2 = \max \lbrace \boldsymbol{Y}_{2,1},\boldsymbol{Y}_{2,2}, \cdots, \boldsymbol{Y}_{2,n}\rbrace$. The following result establishes sufficient conditions for the usual stochastic order between $\boldsymbol{M}_1$ and $\boldsymbol{M}_2$.
\begin{corollary}
If $\boldsymbol{\mu}_2 + \beta(\mathbf{z}) \boldsymbol{\delta}_2 \geq \boldsymbol{\mu}_1 + \beta(\mathbf{z}) \boldsymbol{\delta}_1$ for all $\mathbf{z}$ and $\boldsymbol{\Sigma}_1 = \boldsymbol{\Sigma}_2$, then $\boldsymbol{M}_1 \leq_{st} \boldsymbol{M}_2$.
\begin{proof}
Conditions $\boldsymbol{\mu}_2 + \beta(\mathbf{z}) \boldsymbol{\delta}_2 \geq \boldsymbol{\mu}_1 + \beta(\mathbf{z}) \boldsymbol{\delta}_1$ for all $\mathbf{z}$ and $\boldsymbol{\Sigma}_1 = \boldsymbol{\Sigma}_2$ imply $\boldsymbol{Y}_1 \leq_{st} \boldsymbol{Y}_2$ by Theorem \ref{st-dimn}. Notice that $g(\mathbf{x}) = \max \lbrace x_1, x_2, \cdots, x_n \rbrace$ is increasing, so $f \circ g$ is increasing for all increasing $f$ as well. Then we have $Ef \circ g(\boldsymbol{Y}_1) \leq Ef \circ g(\boldsymbol{Y}_2)$, which implies $\boldsymbol{M}_1 \leq_{st} \boldsymbol{M}_2$.
\end{proof}
\end{corollary}
For a random vector $\boldsymbol{X} = (\boldsymbol{X}_1, \boldsymbol{X}_2, \cdots, \boldsymbol{X}_n)$, the Gini index is defined as
\begin{equation} \label{gini-def}
G_n(\boldsymbol{X}) = \frac{1}{n^2} \sum_{i,j=1}^n |\boldsymbol{X}_i - \boldsymbol{X}_j|,
\end{equation}
where $\boldsymbol{X}_i$ denote the $i$-th component of $\boldsymbol{X}$. Gini index is a well-known tool in economics used for measuring income inequality. In insurance, the Gini index and its modifications have been used to compare the riskiness of different insurance portfolios. We remark that there are various definitions of Gini index and the definition (\ref{gini-def}) we follow is the one defined in \cite{Samanthi2016}.
\begin{corollary}
Let $\boldsymbol{Y}_1,\boldsymbol{Y}_2$ follow (\ref{Y1Y2-assumption}), where $\boldsymbol{\Sigma}_1$ and $\boldsymbol{\Sigma}_2$ are correlation matrices and $n \in \lbrace 2, 3 \rbrace$. Assume $\boldsymbol{\mu}_1 = \boldsymbol{\mu}_2$, $\boldsymbol{\delta}_1 = \boldsymbol{\delta}_2$ and $\sigma_{1,ij} \leq \sigma_{2,ij}$ for all $1 \leq i < j \leq n$, then $G_n(\boldsymbol{Y}_2) \leq_{icx} G_n(\boldsymbol{Y}_1)$.
\end{corollary}
\begin{proof} For $n \in \lbrace 2, 3 \rbrace$, the function $-u(G_n(\mathbf{x})): \mathbb{R}^n \to \mathbb{R}$ is supermodular for any increasing and convex function $u: \mathbb{R} \to \mathbb{R}$; see Lemma 3.2 of \cite{Samanthi2016}. Given $\boldsymbol{\mu}_1 = \boldsymbol{\mu}_2$, $\boldsymbol{\delta}_1 = \boldsymbol{\delta}_2$ and $\sigma_{1,ij} \leq \sigma_{2,ij}$ for all $1 \leq i < j \leq n$, then $\boldsymbol{Y}_1 \leq_{sm} \boldsymbol{Y}_2$ can be established by applying Corollary \ref{sm2}. As a result, $E[u(G_n(\boldsymbol{Y}_2))] \leq E[u(G_n(\boldsymbol{Y}_1))]$, i.e., $G_n(\boldsymbol{Y}_2) \leq_{icx} G_n(\boldsymbol{Y}_1)$.
\end{proof}
\section{Concluding Remarks} \label{remarks}
We have introduced the so-called class of generalized location-scale mixture of elliptical distributions by incorporating the skewness. This class develops a mathematically tractable extension of the well-known multivariate elliptical distributions, and some other types of multivariate distributions studied in the literature. We further derived some sufficient and/or necessary conditions for various integral stochastic orderings of different random vectors following the GLSE distributions. Some useful practical results are provided in this paper.

To conclude the article, we discuss several interesting topics for future study. First, the assumptions for density generator $g$ proposed here are not strict but can still be simplified, and it is of low probability that the results in this paper can be derived with no prior assumptions for $g$. Then finding sufficient assumptions for $g$ should be a challenging and interesting topic for future study. Second, it will naturally be of interest to further generalize results established in this paper to some other families of distributions such as skew-elliptical distributions.

\section*{Acknowledgements}
The authors thank the anonymous referees and the editor for their helpful comments and suggestions, which have led to the improvement of this paper. Yiying Zhang acknowledges the National Natural Science Foundation of China (No. 12101336). Chuancun Yin acknowledges the National Natural Science Foundation of China (No. 12071251).

\appendix

\section{The proof of Remark \ref{ass-satis}}
In this section, we prove that all the density generators presented in Table \ref{table1} follow Assumptions \ref{ass1} and \ref{ass2}. Let $g^1= \left(1+\frac{u}{m}\right)^{-(n+m)/2}$, where $m$ is a positive integer; $g^2(u) = {\rm exp}\left( -\frac{1}{s}(u)^{s/2}\right)$, where $s>1$; $g^3(u) = {\rm exp}\left(-u\right)\left(1+{\rm exp}\left(-u\right)\right)^{-2}$. It is obvious that Cauchy distribution is a special case of Student distribution as normal distribution and Laplace distribution are special cases of exponential power distribution, so we just need to prove the aforementioned three density generators follow Assumptions \ref{ass1} and \ref{ass2}.
\begin{proof}
For $g^1$, we have
\begin{equation} \nonumber
\begin{split}
\lim_{t \to \pm \infty} \frac{\sigma_1}{\sigma_2} \frac{g^1(t_2^2)}{g^1(t_1^2)} &= \lim_{t \to \pm \infty} \frac{\sigma_1}{\sigma_2} \left(\frac{m+t_2^2}{m+t_1^2}\right)^{-\frac{m+1}{2}} = \frac{\sigma_1}{\sigma_2} \left(\lim_{t \to \pm \infty} \frac{m+t_2^2}{m+t_1^2}\right)^{-\frac{m+1}{2}}\\  &= \left(\frac{\sigma_2}{\sigma_1}\right)^m \neq 1 .
\end{split}
\end{equation}
If $\sigma_1 > \sigma_2$, then $\left(\frac{\sigma_2}{\sigma_1}\right)^m < 1$.
For $g^2$, we have
\begin{equation} \nonumber
\begin{split}
\lim_{t \to \pm \infty} \frac{\sigma_1}{\sigma_2} \frac{g^2(t_2^2)}{g^2(t_1^2)} &=\lim_{t \to \pm \infty} \frac{\sigma_1}{\sigma_2} {\rm exp}\left( \frac{1}{s} \left(t_1^s - t_2^s\right)\right) = \lim_{t \to \pm \infty} \frac{\sigma_1}{\sigma_2} {\rm exp}\left( \frac{1}{s} \left(\frac{1}{\sigma_1^s} - \frac{1}{\sigma_2^s}\right)t^s\right).
\end{split}
\end{equation}
If $\sigma_1 > \sigma_2$, then $\lim_{t \to \pm \infty} \frac{\sigma_1}{\sigma_2} \frac{g^2(t_2^2)}{g^2(t_1^2)}$ goes to zero otherwise goes to infinity.
\par For $g^3$, we have
\begin{equation} \nonumber
\lim_{t \to \pm \infty} \frac{\sigma_1}{\sigma_2} \frac{g^3(t_2^2)}{g^3(t_1^2)} =\lim_{t \to \pm \infty} \frac{\sigma_1}{\sigma_2} \frac{{\rm exp}\left(-t_2^2\right)}{{\rm exp}\left(-t_1^2\right)} \frac{\left(1+{\rm exp}\left(-t_1^2\right)\right)^2}{\left(1+{\rm exp}\left(-t_2^2\right)\right)^2}.
\end{equation}
We have
\begin{equation} \nonumber
\lim_{t \to \pm \infty} \frac{\left(1+{\rm exp}\left(-t_1^2\right)\right)^2}{\left(1+{\rm exp}\left(-t_2^2\right)\right)^2} = 1.
\end{equation}
If $\sigma_1 > \sigma_2$, then $\lim_{t \to \pm \infty} {\rm exp}\left(t_1^2 - t_2^2 \right)$ goes to zero otherwise goes to infinity. So $\lim_{t \to \pm \infty} \frac{\sigma_1}{\sigma_2} \frac{g^3(t_2^2)}{g^3(t_1^2)}$ behaves the same way.
\end{proof}

\bigskip


\begin{thebibliography}{}
\bibitem[Abdi, Balakrishnan, and Jamalizadeh(2020)]{madadi2020family}
Madadi, M., Balakrishnan, N., \& Jamalizadeh, A. (2021). Family of mean-mixtures of multivariate normal distributions: properties, inference and assessment of multivariate skewness. Journal of Multivariate Analysis, 181, 104679.

\bibitem[Adcock and Shutes(2012)]{Adcock2012}
Adcock, C. J., \& Shutes, K. (2012). On the multivariate extended skew-normal, normal-exponential, and normal-gamma distributions. Journal of Statistical Theory and Practice, 6(4), 636-664.

\bibitem[Amiri, Izadkhah and Jamalizadeh(2020)]{Amiri2020Linear}
Amiri, M., Izadkhah, S., \& Jamalizadeh, A. (2020). Linear orderings of the scale mixtures of the multivariate skew-normal distribution. Journal of Multivariate Analysis, 179, 104647.

\bibitem[Amiri and Balakrishnan(2022)]{Amiri2022Hessian}
Amiri, M., \& Balakrishnan, N. (2022). Hessian and increasing-Hessian orderings of scale-shape mixtures of multivariate skew-normal distributions and applications. Journal of Computational and Applied Mathematics, 402, 113801.

\bibitem[Ansari and R{\"u}schendorf(2020)]{Ansari2020Ordering}
Ansari, J., \& R{\"u}schendorf, L. (2021). Ordering results for elliptical distributions with applications to risk bounds. Journal of Multivariate Analysis, 182, 104709.

\bibitem[Arlotto and Scarsini(2009)]{arlotto2009hessian}
Arlotto, A., \& Scarsini, M. (2009). Hessian orders and multinormal distributions. Journal of multivariate analysis, 100(10), 2324-2330.

\bibitem[Arnold et al.(2002)]{Arnold2002}
Arnold, B. C., Beaver, R. J., Azzalini, A., Balakrishnan, N., Bhaumik, A., Dey, D. K., Cuadras, C. M. \& Sarabia, J. M. (2002). Skewed multivariate models related to hidden truncation and/or selective reporting. Test, 11(1), 7-54.

\bibitem[Arslan(2008)]{Arslan2008}
Arslan, O. (2008). An alternative multivariate skew-slash distribution. Statistics \& Probability Letters, 78(16), 2756-2761.

\bibitem[Arslan(2015)]{Arslan2015}
Arslan, O. (2015). Variance-mean mixture of the multivariate skew normal distribution. Statistical Papers, 56(2), 353-378.

\bibitem[Azzalini(1985)]{azzalini1985class}
Azzalini, A. (1985). A class of distributions which includes the normal ones. Scandinavian journal of statistics, 12(2), 171-178.

\bibitem[Azzalini and Dalla Valle(1996)]{azzalini1996multivariate}
Azzalini, A., \& Valle, A. D. (1996). The multivariate skew-normal distribution. Biometrika, 83(4), 715-726.

\bibitem[Azzalini(2005)]{Azzalini2005}
Azzalini, A. (2005). The skew‐normal distribution and related multivariate families. Scandinavian journal of statistics, 32(2), 159-188.

\bibitem[Barndorff-Nielsen and Blaesild(1981)]{Barndorff1981}
Barndorff-Nielsen, O., \& Blaesild, P. (1981). Hyperbolic distributions and ramifications: Contributions to theory and application. In Statistical distributions in scientific work (pp. 19-44). Springer, Dordrecht.

\bibitem[Barndorff-Nielsen, Kent and S{\o}rensen(1982)]{Barndorff1982Normal}
Barndorff-Nielsen, O., Kent, J., \& S{\o}rensen, M. (1982). Normal variance-mean mixtures and z distributions. International Statistical Review/Revue Internationale de Statistique, 50(2), 145-159.

\bibitem[B{\"a}uerle(1997)]{Bauerle1997}
B{\"a}uerle, N. (1997). Inequalities for stochastic models via supermodular orderings. Stochastic Models, 13(1), 181-201.

\bibitem[B{\"a}uerle(2014)]{Bauerle2014}
B{\"a}uerle, N., \& Bayraktar, E. (2014). A note on applications of stochastic ordering to control problems in insurance and finance. Stochastics An International Journal of Probability and Stochastic Processes, 86(2), 330-340.

\bibitem[Branco and Dey(2001)]{Branco2001}
Branco, M. D., \& Dey, D. K. (2001). A general class of multivariate skew-elliptical distributions. Journal of Multivariate Analysis, 79(1), 99-113.

\bibitem[Davidov and Peddada(2013)]{Davidov2013linear}
Davidov, O., \& Peddada, S. (2013). The linear stochastic order and directed inference for multivariate ordered distributions. Annals of statistics, 41(1), 1-40.

\bibitem[De la Cal and Carcamo(2006)]{Cal2006stochastic}
De la Cal, J., \& Carcamo, J. (2006). Stochastic orders and majorization of mean order statistics. Journal of Applied Probability, 43(3), 704-712.

\bibitem[Denuit and M{\"u}ller(2002)]{denuit2002smooth}
Denuit, M., \& M{\"u}ller, A. (2002). Smooth generators of integral stochastic orders. The Annals of Applied Probability, 12(4), 1174-1184.

\bibitem[Denuit et al.(2006)]{denuit2006actuarial}
Denuit, M., Dhaene, J., Goovaerts, M., \& Kaas, R. (2006). Actuarial Theory For Dependent Risks: Measures, Orders and Models. John Wiley \& Sons, Chichester.

\bibitem[Dey and Liu(2005)]{Dey2005}
Dey, D. K., \& Liu, J. (2005). A new construction for skew multivariate distributions. Journal of multivariate analysis, 95(2), 323-344.

\bibitem[Ding and Zhang(2004)]{ding2004some}
Ding, Y., \& Zhang, X. (2004). Some stochastic orders of Kotz-type distributions. Statistics \& probability letters, 69(4), 389-396.

\bibitem[F{\'a}bi{\'a}n, Mitra and Roman(2011)]{fabian2011Processing}
F{\'a}bi{\'a}n, C. I., Mitra, G., \& Roman, D. (2011). Processing second-order stochastic dominance models using cutting-plane representations. Mathematical Programming, 130(1), 33-57.

\bibitem[Fang, Kotz and Ng(1990)]{Fang1990}
Fang, K. T., Kotz, S. \& Ng, K. W. (1990). Symmetric Multivariate and Related Distributions, Chapman \& Hall, London.

\bibitem[Gupta, Varga and Bodnar(2013)]{Gupta2013}
Gupta, A. K., Varga, T., \& Bodnar, T. (2013). Elliptically Contoured Models in Statistics and Portfolio Theory. Springer, New York.

\bibitem[Gupta et al.(1971)]{Gupta1972}
Gupta, S. D., Eaton, M. L., Olkin, I., Perlman, M., Savage, L. J., \& Sobel, M. (1971). Inequalities on the probability content of convex regions for elliptically contoured distributions. In: Sixth Berkeley Symposium on Probability and Statistics, 241-265.

\bibitem[Jamali, Amiri and Jamalizadeh(2021)]{jamali2020comparison}
Jamali, D., Amiri, M., \& Jamalizadeh, A. (2021). Comparison of the multivariate skew-normal random vectors based on the integral stochastic ordering. Communications in Statistics-Theory and Methods, 50(22), 5215-5227.

\bibitem[Jamali et al.(2020)]{jamali2020integral}
Jamali, D., Amiri, M., Jamalizadeh, A., \& Balakrishnan, N. (2020). Integral stochastic ordering of the multivariate normal mean-variance and the skew-normal scale-shape mixture models. Statistics, Optimization \& Information Computing, 8(1), 1-16.

\bibitem[Jones(2004)]{Jones2004}
Jones, M. C. (2004). Families of distributions arising from distributions of order statistics. Test, 13(1), 1-43.

\bibitem[Kelker(1970)]{Kelker1970Distribution}
Kelker, D. (1970). Distribution theory of spherical distributions and a location-scale parameter generalization. Sankhy\={a}: The Indian Journal of Statistics, Series A, 419-430.


\bibitem[Kim and Kim(2019)]{kim2019Tail}
Kim, J. H., \& Kim, S. Y. (2019). Tail risk measures and risk allocation for the class of multivariate normal mean-variance mixture distributions. Insurance: Mathematics and Economics, 86, 145-157.

\bibitem[Landsman and Tsanakas(2006)]{landsman2006stochastic}
Landsman, Z., \& Tsanakas, A. (2006). Stochastic ordering of bivariate elliptical distributions. Statistics \& Probability Letters, 76(5), 488-494.


\bibitem[McNeil, Frey and Embrechts(2015)]{McNeil2015}
McNeil, A. J., Frey, R., \& Embrechts, P. (2015). Quantitative Risk Management: Concepts, Techniques and Tools-revised Edition. Princeton University Press, New Jersey.

\bibitem[M{\"u}ller(1997)]{muller1997stochastic}
MM{\"u}ller, A. (1997). Stochastic orders generated by integrals: a unified study. Advances in Applied Probability, 29(2), 414-428.

\bibitem[M{\"u}ller(2001)]{muller2001stochastic}
M{\"u}ller, A. (2001). Stochastic ordering of multivariate normal distributions. Annals of the Institute of Statistical Mathematics, 53(3), 567-575.


\bibitem[M{\"u}ller and Scarsini(2000)]{muller2000some}
M{\"u}ller, A., \& Scarsini, M. (2000). Some remarks on the supermodular order. Journal of multivariate analysis, 73(1), 107-119.


\bibitem[M{\"u}ller and Stoyan(2002)]{muller2002comparison}
M{\"u}ller, A., \& Stoyan D. (2002). Comparison Methods for Stochastic Models and Risks, Wiley, New York.

\bibitem[Negarestani et al.(2019)]{Negarestani2019}
Negarestani, H., Jamalizadeh, A., Shafiei, S., \& Balakrishnan, N. (2019). Mean mixtures of normal distributions: properties, inference and application. Metrika, 82(4), 501-528.


\bibitem[Pan, Qiu and Hu(2016)]{pan2016stochastic}
Pan, X., Qiu, G., \& Hu, T. (2016). Stochastic orderings for elliptical random vectors. Journal of Multivariate Analysis, 148, 83-88.
\bibitem[Pu, Balakrishnan and Yin(2022)]{pu2022identity}
Pu, T., Balakrishnan, N., \& Yin, C. (2022). An identity for expectations and characteristic function of matrix variate skew-normal distribution with applications to associated stochastic orderings. Communications in Mathematics and Statistics, 1-19.

\bibitem[R{\"u}schendorf(1980)]{Ruschendorf}
R{\"u}schendorf, L. (1980). Inequalities for the expectation of $\Delta$-monotone functions. Zeitschrift f{\"u}r Wahrscheinlichkeitstheorie und verwandte Gebiete, 54(3), 341-349.

\bibitem[Samanthi, Wei and Brazauskas(2016)]{Samanthi2016}
Samanthi, R. G. M., Wei, W., \& Brazauskas, V. (2016). Ordering Gini indexes of multivariate elliptical risks. Insurance: Mathematics and Economics, 68, 84-91.
\bibitem[Scarsini(1998)]{Scarsini1998}
Scarsini, M. (1998). Multivariate convex orderings, dependence, and stochastic equality. Journal of Applied Probability, 35(1), 93-103.


\bibitem[Shaked and Shanthikumar(1994)]{shaked1994}
Shaked, M. \&  Shanthikumar J.G. (1994). Stochastic Orders and Their Applications. Academic Press, London.

\bibitem[Shaked and Shanthikumar(2007)]{shaked2007stochastic}
Shaked, M. \& Shanthikumar J.G. (2007). Stochastic orders. Springer, New York,.

\bibitem[Slepian(1962)]{Slepian1962}
Slepian, D. (1962). The one-sided barrier problem for Gaussian noise. Bell System Technical Journal, 41(2), 463-501.

\bibitem[Simaan(1993)]{simaan1993}
Simaan, Y. (1993). Portfolio selection and asset pricing—three-parameter framework. Management Science, 39(5), 568-577.


\bibitem[Wang, Boyer and Genton(2004)]{Wang2004}
Wang, J., Boyer, J., \& Genton, M. G. (2004). A skew-symmetric representation of multivariate distributions. Statistica Sinica, 14(4), 1259-1270.

\bibitem[Tong(2014)]{tong}
Tong, Y. L. (2014). Probability inequalities in multivariate distributions. Academic Press, New York.
\bibitem[Topkis(1988)]{topkis}
Topkis, D. M. (1988). Supermodularity and Complementarity. Princeton University Press, New Jersey.

\bibitem[Yin(2019)]{yin2019stochastic}
Yin, C. (2021). Stochastic orderings of multivariate elliptical distributions. Journal of Applied Probability, 58(2), 551-568.

\bibitem[Zhang and Zhao(2015)]{zhang2015}
Zhang, Y., \& Zhao, P. (2015). Comparisons on aggregate risks from two sets of heterogeneous portfolios. Insurance: Mathematics and Economics, 65, 124-135.
\bibitem[Zhang, Zhao and Cheung(2019)]{zhang2019}
Zhang, Y., Zhao, P., \& Cheung, K. C. (2019). Comparisons of aggregate claim numbers and amounts: a study of heterogeneity. Scandinavian Actuarial Journal, 2019(4), 273-290.
\bibitem[Zuo and Yin(2021)]{zuo2021}
Zuo B. \& Yin C. (2021). Tail conditional risk measures for location-scale mixture of elliptical distributions, Journal of Statistical Computation and Simulation, 91(17), 3653-3677.

\end{thebibliography}
\end{document}